\pdfoutput=1
\documentclass[11pt,a4paper,reqno,twoside]{amsart}
\usepackage[utf8]{inputenc}
\usepackage[english]{babel}

\usepackage[svgnames]{xcolor}
\usepackage{amssymb}
\usepackage[linktoc=all,colorlinks=true,linkcolor=Green,citecolor=Orange,urlcolor=DarkBlue]{hyperref}
\hypersetup{pdftitle={}, pdfauthor={D. Hundertmark, Y.-R. Lee, T. Ried, V. Zharnitsky}, pdfsubject={35Q55 (Primary); 35Q50 (Secondary)}, pdfkeywords={Saturated Nonlinearities, Nonlocal NLS, Solitary Waves, Nonlocal Variational Problems}}

\usepackage{newtxtext}
\usepackage{newtxmath}
\usepackage{microtype}
\usepackage{dsfont}
\usepackage[inline]{enumitem}
\newcommand{\R}{\mathord{\mathbb{R}}}
\newcommand{\C}{\mathord{\mathbb{C}}}
\newcommand{\N}{\mathord{\mathbb{N}}}

\newcommand{\I}{\mathord{\mathrm{i}}}
\newcommand{\1}{\mathord{\mathds{1}}}

\renewcommand{\Re}{\mathrm{Re}\,}
\renewcommand{\coloneq}{~\raise.35pt\hbox{:}\kern-3.5pt=}
\renewcommand{\eqcolon}{=\kern-3.5pt\raise.35pt\hbox{:}~}

\theoremstyle{plain}
\newtheorem{theorem}{Theorem}[section]
\newtheorem{corollary}[theorem]{Corollary}
\newtheorem{proposition}[theorem]{Proposition}
\newtheorem{lemma}[theorem]{Lemma}
\newtheorem*{theorem*}{Theorem}
\newtheorem*{conjecture*}{Conjecture}

\theoremstyle{definition}
\newtheorem{definition}[theorem]{Definition}
\newtheorem{remark}[theorem]{Remark}
\newtheorem*{remarks*}{Remarks}
\providecommand{\mr}[1]{\href{http://www.ams.org/mathscinet-getitem?mr=#1}{MR~#1}}
\providecommand{\zbl}[1]{\href{https://zbmath.org/?q=an:#1}{Zbl~#1}}

\providecommand{\doi}[2]{\href{https://doi.org/#1}{#2}}
\title[Solitary waves in nonlocal NLS with saturated nonlinearities]
      {Solitary waves in nonlocal  NLS with dispersion averaged saturated nonlinearities}
\author[D. Hundertmark, Y.-R. Lee, T. Ried, V. Zharnitsky]{}
\subjclass[2010]{Primary: $35Q55$; Secondary: $35Q50$.}
\keywords{Saturated Nonlinearities, Nonlocal NLS, Solitary Waves, Nonlocal Variational Problems}

\email{\url{dirk.hundertmark@kit.edu}}
\email{\url{younglee@sogang.ac.kr}}
\email{\url{tobias.ried@kit.edu}}
\email{\url{vz@math.uiuc.edu}}
\thanks{\copyright~2017 by the authors. Faithful reproduction of this article, in its entirety, by any means is permitted for non-commercial purposes}
\begin{document}
\maketitle

\centerline{\scshape Dirk Hundertmark}
\medskip
{\footnotesize
 \centerline{Institute for Analysis, Karlsruhe Institute of Technology (KIT)}
 \centerline{Englerstra{\ss}e 2, 76131 Karlsruhe, Germany}
}

\medskip

\centerline{\scshape Young-Ran Lee
}
\medskip
{\footnotesize
 \centerline{Department of Mathematics, Sogang University}
 \centerline{35 Baekbeom-ro, Mapo-gu, Seoul 04107, South Korea}
}

\medskip

\centerline{\scshape Tobias Ried
}
\medskip
{\footnotesize
 \centerline{Institute for Analysis, Karlsruhe Institute of Technology (KIT)}
 \centerline{Englerstra{\ss}e 2, 76131 Karlsruhe, Germany}
}

\medskip

\centerline{\scshape Vadim Zharnitsky
}
\medskip
{\footnotesize
 \centerline{Department of Mathematics, University of Illinois at Urbana-Champaign}
 \centerline{1409 W. Green Street, Urbana, Illinois 61801-2975, USA}
}

\bigskip

\begin{abstract}

A nonlinear Schr\" odinger equation (NLS) with dispersion averaged  nonlinearity of saturated type is considered.
Such a nonlocal NLS is of integro-differential type and it arises naturally in modeling  fiber-optics communication  systems with periodically varying dispersion profile (dispersion management).
The associated constrained variational principle is shown to posses a ground state solution by constructing a convergent minimizing sequence through the application of a method  similar to the classical concentration compactness principle of Lions. One of the obstacles  in applying this variational approach is  that a saturated nonlocal nonlinearity does not satisfy uniformly  the so-called strict sub-additivity condition. This is overcome by applying a special version of Ekeland's variational principle.
\end{abstract}

\setcounter{tocdepth}{1}
\tableofcontents
\section{Introduction and main results}\label{sec:introduction}

We consider a nonlocal variational problem  related to an averaged nonlinear Schr\"odinger equation as it  arises in the context of fiber optical communication systems  with periodically varying dispersion \emph{(dispersion management)}.  After application of an averaging approximation, the nonlinearity  is averaged over the dispersive action which leads to a nonlocal  problem.
For Kerr type  nonlinearities, solitary waves have been constructed in \cite{GT96,AB98,ZGJT01,Kun04,HL12}, and for more general nonlinearities in \cite{CHL15}.
The main novelty of this article is that saturating nonlinearities are allowed, which poses significant problems in proving the existence of ground state solutions in the nonlocal variational problem.
On a more technical level, we would like to emphasize that we impose weaker than usual differentiability properties on the energy functional defined below, more precisely, we do not assume the energy functional to be continuously differentiable.

While  in the local case a saturable  nonlinearity is often helpful by  creating  more favorable conditions for the existence of ground states, e.g. by arresting collapse in the supercritical regime, in the nonlocal case saturation presents difficulties in satisfying sub-additvity condition.

From a physics viewpoint, saturated nonlinearities are relevant  in modeling optical waves in nonlinear materials as the Kerr nonlinearity, which is cubic for low intensities, saturates
for large fields and approaches a regime with constant refraction index. One needs to assume some form of saturation of the nonlinearity $P(u)=p(|u|)u$ with the most common law being
\[
p(|u|) = \frac{|u|^2}{1+ \sigma |u|^2},
\]
where $p$ corresponds to the intensity dependent refraction coefficient.
In such models, the corresponding term in the Hamiltonian functional,  which we call {\em nonlinearity potential}, is given by
\[
V(a)  = \frac{a^2}{2\sigma} - \frac{1}{2\sigma^2}\log(1+ \sigma a^2),
\]
where $V'(a) = P(a)$.
Another natural modification with similar behavior is given by the nonlinearity potential
$
V(a) = a^4/(1+  \sigma a^2).
$
We will actually consider a much broader class of saturated nonlinearity potentials which include the above two as very special cases, but first we discuss the local saturated NLS.

The presence of solitary waves in local NLS with saturated nonlinearity has been addressed in several studies, see e.g.  \textsc{Gatz-Herrmann} \cite{GH91},
\textsc{Usman-Osman-Tilley} \cite{UOT98} and references therein.
Their results show  that solitary wave solutions can be obtained numerically and sometimes analytically using phase space analysis and  one may also observe two-state solitons. In particular, one may observe  the bi-stability phenomenon.

 In this paper, we address the question of existence of at least one solitary wave, for nonlocal NLS, which turns out already a challenging task as we cannot construct solitons by
 using phase space analysis as in the local case.

 As we employ variational methods, the obtained solitary wave is automatically  a ground state solution. It is anticipated that  multiple solitary waves may also exist in nonlocal case, but one needs to use different methods to address this question.

Saturable nonlinearities have also been considered in the context of coupled local nonlinear Schr\"odinger systems, where the existence of stationary solutions was also proved by variational and bifurcation techniques, see \cite{dAMP13,JT02,Man16}. \\

The direct application of the approach of \cite{CHL15} to saturating nonlinearities does not work because of the lack of strict sub-additivity of the energy in this case. The main idea is to construct a modified minimizing sequence with a uniform $L^{\infty}$ bound, which prevents the minimizing sequence from reaching the saturation regime.

The problem we are addressing in the work at hand is related to the existence of breather-type solutions of the dispersion managed one-dimensional nonlinear Schr\" odinger equation
\begin{align}\label{eq:NLS-gen}
	\mathrm{i} \partial_t u = -d(t) \, \partial_x^2 u - p(|u|)u,
\end{align}
where the dispersion $d(t) = \epsilon^{-1} d_0(t/\epsilon) + d_{\mathrm{av}}$ is parametrically modulated. The constant  part of the group velocity dispersion $d_{\mathrm{av}}$,  assumed to be nonnegative $d_{\mathrm{av}} \geq 0$, denotes the average component (residual dispersion) and $d_0$ its $L$-periodic mean zero part, the most basic example being $d_0 = \1_{[0,1)} - \1_{[1, 2)}$ for $L=2$. $P(u)=p(|u|)u$ is the nonlinear interaction due to the polarizability of the optical fiber.

Let $T_r = \mathrm{e}^{\I r \partial_{x}^2}$ be the free Schr\"odinger evolution in one space dimension and write $u = T_{D(t/\epsilon)} v$, with $D(t) \coloneq \int_0^t d_0(r)\,\mathrm{d}r$. Then \eqref{eq:NLS-gen} is equivalent to
\begin{align*}
	\mathrm{i} \partial_t v = -d_{\mathrm{av}} \, \partial_x^2 v - T^{-1}_{D(t/\epsilon)} \left[ P(T_{D(t/\epsilon)} v) \right].
\end{align*}
In the limit of small $\epsilon$, and averaging over the fast dispersion action, one obtains an averaged dispersion managed nonlinear Schr\"odinger equation\footnote{The solutions of the averaged equation and of the original one turn out to be $\epsilon$ close on the time scale of order $\epsilon^{-1}$.  It can be shown by developing an appropriate averaging theory similar to \cite{ZGJT01}  but we do not pursue this direction here.},
\begin{align*}
\mathrm{i} \partial_{\tau} v &= -d_{\mathrm{av}} \, \partial_x^2 v -  \frac{1}{L} \int_0^L T^{-1}_{D(r)} \left[ P(T_{D(r)} v)  \right] \,\mathrm{d}r \\
	&= -d_{\mathrm{av}} \, \partial_x^2 v - \int_{\R} T^{-1}_r \left[ P(T_r v)  \right] \,\mu(\mathrm{d}r),
\end{align*}
where $\mu$ is the image of the uniform measure on $[0,L]$ under $D$.  In the model example $d_0 = \1_{[0,1)} - \1_{[1, 2)}$ from above, the measure $\mu$ has a density $\psi = \1_{[0,1]}$ with respect to Lebesgue measure on $\R$.
More generally, under physically reasonable assumptions, the probability measure $\mu$ has compact support and is absolutely continuous with respect to Lebesgue measure, with density $\psi$ in suitable $L^p$ spaces, see \cite[Lemma 1.4]{HL12} for details.

Standing wave solutions of the averaged DM NLS of the form $v(t,x) = \mathrm{e}^{-\mathrm{i}\omega t} f(x)$ are solutions of the nonlinear and nonlocal eigenvalue equation (Dispersion management equation, \textsc{Gabitov-Turitsyn} \cite{GT96})
	\begin{align}\label{eq:DMequation}
		\omega f = - d_{\mathrm{av}} \, f'' - \int_{\R} T^{-1}_r \left[ P(T_r f)  \right] \,\mu(\mathrm{d}r)
	\end{align}
The DM equation \eqref{eq:DMequation} is of variational type  as it can be considered as an Euler-Lagrange equation for an appropriate variational principle and weak solutions can be found as stationary points of a suitable energy functional, see \eqref{eq:varprob}.

\subsection{The variational problems}
We study the existence of minimizers for the nonlinear and nonlocal variational problems
\begin{align}\label{eq:varprob}
	E_{\lambda}^{d_{\mathrm{av}}} \coloneq \inf_{f\in\mathcal{S}_{\lambda}^{d_{\mathrm{av}}}} H(f),
\end{align}
where
\begin{align*}
\mathcal{S}_{\lambda}^{d_{\mathrm{av}}} &= \left\{u\in H^1(\R;\C): \|u\|_2^2 = \lambda\right\}  \quad \text{for}\quad \lambda>0,\, d_{\mathrm{av}} > 0,\\
\mathcal{S}_{\lambda}^{0} &= \left\{u\in L^2(\R;\C): \|u\|_2^2 = \lambda\right\} \ \quad \text{for}\quad \lambda>0,\, d_{\mathrm{av}} = 0, 	
\end{align*}
where $L^2(\R;\C)$, respectively $H^1(\R;\C)$ are the standard $L^2$, respectively Sobolev spaces for complex-valued functions. We will denote the usual inner product on $L^2(\R;\C)$ by $\langle f, g \rangle = \int_{\R} \overline{f} g\,\mathrm{d}x$, but consider $L^2(\R;\C)$ as \emph{real} Hilbert space with the inner product $\mathrm{Re}\langle\cdot, \cdot\rangle$, which induces the same topology on $L^2(\R;\C)$. $H^1(\R;\C)$ will be equipped with the corresponding $L^2$-inner product.
We will also use the standard $L^p$ norms denoted by $\|\cdot\|_p$.
The energy functional is given by
\begin{align}\label{eq:hamiltonian}
H(f) \coloneq \frac{d_{\mathrm{av}}}{2} \|f'\|_2^2 - N(f).
\end{align}
for $f\in \mathcal{H}^{d_\mathrm{av}}$, where we set $\mathcal{H}^{d_\mathrm{av}}= H^1(\R;\C)$ if $d_{\mathrm{av}}>0$ and $\mathcal{H}^0=L^2(\R;\C)$ for convenience.
The nonlocal nonlinearity is given by
\begin{align}\label{eq:nonlinearity}
N(f) \coloneq \iint_{\R^2} V(|T_r f(x)|) \, \mathrm{d}x\, \psi(r)\,\mathrm{d}r
\end{align}
for some suitable nonlinear potential $V: [0, \infty) \to [0, \infty)$. Here,  the function $\psi$ is the density of some compactly supported probability measure and will be assumed to lie in appropriate $L^p$ spaces.

We assume that $V$ satisfies the following assumptions:
\begin{enumerate}[label=(\textbf{A\arabic*})]
\item \label{ass:growth} $V$ is continuous on $[0,\infty)$ and continuously differentiable on $(0,\infty)$ with $V(0)=0$. There exist $2 \leq \gamma_1 \leq \gamma_2 < \infty$ such that
	\begin{align*}
		|V'(a)|\lesssim a^{\gamma_1-1} + a^{\gamma_2-1} \quad \text{for all} \quad a>0.
	\end{align*}
\item \label{ass:saturation} There exists a continuous function $\kappa: [0, \infty) \to [2, \infty)$ with $\kappa > 2$ on \emph{compact} intervals, such that for all $a>0$,
	\begin{align*}
		V'(a)a \geq \kappa(a) V(a).
	\end{align*}
\item \label{ass:strictpos} There exists $a_*>0$ such that $V(a_*)>0$.
\end{enumerate}

\begin{remark}
Assumption \ref{ass:saturation} allows for \emph{saturation} of the potential $V$ in the sense that $\frac{V'(a)a}{V(a)} \to 2$ as $a\to \infty$. This is in contrast to the typically assumed Ambrosetti-Rabinowitz condition \cite{AR73}
	\begin{align}\label{eq:AR}
		V'(a) a \geq \kappa V(a) \quad \text{for all} \quad a>0
	\end{align}
with $\kappa>2$, which fails in the limit $a\to\infty$ for the saturated nonlinearities. In \cite{CHL15}, the Ambrosetti-Rabinowitz condition \eqref{eq:AR} was crucial in proving strict sub-additivity of the variational problem.
\end{remark}

Under assumptions \ref{ass:growth}--\ref{ass:strictpos} with appropriate restrictions on $\gamma_1, \gamma_2$, we can show that there exists a threshold for the existence of minimizers. Under the additional assumptions
\begin{enumerate}[label=(\textbf{A\arabic*}), start=4]
\item \label{ass:zerothreshold}
	\begin{enumerate}[leftmargin=10ex]
	\item[$d_{\mathrm{av}}=0$:] There exists $\epsilon>0$ such that $V(a)>0$ for all $0<a\leq \epsilon$.
	\item[$d_{\mathrm{av}}>0$:] There exist $\epsilon>0$ and $2<\gamma_0<6$ such that $V(a)\gtrsim a^{\gamma_0}$ for all $0<a\leq \epsilon$.
	\end{enumerate}
\end{enumerate}
on $V$, minimizers are shown to exist for any $\lambda>0$.

\begin{remark}\label{ass:lipschitz}
If we assume that there exists $\gamma\geq 2$ such that
\begin{align}\label{eq:lipschitz}
	|V'(|z+w|) - V'(|z|)| \lesssim (|w| + |z|)^{\gamma-2} \, |w| \quad \text{for all} \quad w,z \in \C,
\end{align}
the nonlinearity $N: \mathcal{H}^{d_{\mathrm{av}}} \to \R$ is actually a $\mathcal{C}^1$ functional, see Proposition \ref{prop:diff-nonlin}.
	We can work with directional derivatives only though, including the construction of the modified minimizing sequence, so this assumption is not needed for our main theorems.
\end{remark}

\begin{remark}
	If no further restriction is made, we shall assume throughout that the function $\psi$ is a compactly supported, non-negative, and integrable function. Moreover, from \cite[Lemma 1.4]{HL12} we infer that if the periodic dispersion profile $d_0$ changes sign finitely many times and $\int_0^L |d_0(s)|^{1-p}\, \mathrm{d}s <\infty$ for some $p\ge 1$, then the probability density $\psi\in L^p$ and has compact support, so the $L^p$ conditions on $\psi$ in our existence theorems are very reasonable assumptions in view of the applications of our results to nonlinear optics.
\end{remark}

\subsection{Main results}

\begin{theorem}[Existence of DM solitons for zero average dispersion]\label{thm:zero}
Let $d_{\mathrm{av}} = 0$.
Assume $V$ satisfies the conditions \ref{ass:growth}--\ref{ass:strictpos}, with $3\leq \gamma_1\leq \gamma_2 <5$. Assume further that $\psi\geq 0$ is compactly supported and $\psi\in L^{\frac{4}{5-\gamma_2} + \delta}$ for some $\delta >0$.

Then there exists a threshold $\lambda_{\mathrm{cr}}^0 \geq 0$ such that
\begin{enumerate}
	\item if $0<\lambda<\lambda_{\mathrm{cr}}^0$, then $E_{\lambda}^0=0$,
	\item if $\lambda>\lambda_{\mathrm{cr}}^0$, then $-\infty<E_{\lambda}^0<0$ and there exists a minimizer $u\in \mathcal{S}_{\lambda}^{0}\cap L^{\infty}$ of the variational problem \eqref{eq:varprob}. This minimizer is a weak solution of the dispersion management equation
	\begin{equation}\label{eq:DM}
 		\omega f = -\int_{\R} T_r^{-1} \left[ V'(|T_r f|) \tfrac{T_r f}{|T_r f|} \right]\,  \psi\,\mathrm{d}r.
	\end{equation}
	for some Lagrange multiplier $\omega < \frac{2 E_{\lambda}^0}{\lambda} < 0$.
\end{enumerate}

If, in addition, assumption \ref{ass:zerothreshold} holds, then $\lambda_{\mathrm{cr}}^0 = 0$.
\end{theorem}

\begin{theorem}[Existence of DM solitons for positive average dispersion]\label{thm:positive}
Let $d_{\mathrm{av}}>0$.

Assume $V$ satisfies the conditions \ref{ass:growth}--\ref{ass:strictpos}, with $2\leq \gamma_1\leq \gamma_2 <10$. Assume further that $\psi\geq 0$ is compactly supported, with $\psi\in L^{a_{\delta}}$ for some $\delta>0$, where $a_{\delta} \coloneq \max\left\{1, \frac{4}{10-\gamma_2}+\delta\right\}< \infty$.

Then there exists a threshold $\lambda_{\mathrm{cr}}^{d_{\mathrm{av}}} \geq 0$ such that
\begin{enumerate}
	\item if $0<\lambda<\lambda_{\mathrm{cr}}^{d_{\mathrm{av}}}$, then $E_{\lambda}^{d_{\mathrm{av}}}=0$ and there exists no minimizer for the variational problem \eqref{eq:varprob},
	\item if $\lambda>\lambda_{\mathrm{cr}}^{d_{\mathrm{av}}}$, then $-\infty<E_{\lambda}^{d_{\mathrm{av}}}<0$ and there exists a minimizer $u\in \mathcal{S}_{\lambda}^{d_{\mathrm{av}}}$ of the variational problem \eqref{eq:varprob}. This minimizer is a weak solution of the dispersion management equation
	\begin{equation}\label{eq:DMpos}
 		\omega f = -d_{\mathrm{av}} f'' - \int_{\R} T_r^{-1} \left[ V'(|T_r f|) \tfrac{T_r f}{|T_r f|} \right] \, \psi\,\mathrm{d}r.
	\end{equation}
	for some Lagrange multiplier $\omega < \frac{2 E_{\lambda}^{d_{\mathrm{av}}}}{\lambda} < 0$.
\end{enumerate}

If, in addition, assumption \ref{ass:zerothreshold} holds, then $\lambda_{\mathrm{cr}}^{d_{\mathrm{av}}} = 0$.
\end{theorem}

The main ingredient in the proof of existence of minimizers of \eqref{eq:varprob} is the construction of a minimizing sequence which satisfies an additional uniform $L^{\infty}$ bound. This prevents the minimizing sequence from reaching the asymptotic regime, where strict sub-additivity would fail.

While for positive average dispersion $d_{\mathrm{av}}>0$, the uniform $L^{\infty}$ bound is readily provided by the Sobolev embedding $H^1(\R) \hookrightarrow L^{\infty}(\R)$, some work has to be done in the setting of zero average dispersion. More precisely, we will construct a modified minimizing sequence via Ekeland's variational principle, which provides an approximate solution of the DM equation, combined with dispersive estimates on the gradient of the nonlinearity.

A special case of our main results, obtained by a different method, has been reported in \cite{HLRZ17}.

\section{Preparatory and technical remarks}
In this section we review some important properties of the nonlinearity $N$.
Most of these properties are adapted from \cite{CHL15} which the reader may consult for more complete background.
The basic ingredient in most of these estimates is
\begin{lemma}\label{lem:st-boundedness}
Let $f\in L^2(\R)$, $2\leq q \leq 6$, and $\psi\in L^{\frac{4}{6-q}}$. Then
\begin{align} \label{eq:boundedness}
	\|T_r f\|_{L^{q}(\R^2,\mathrm{d}x\,\psi \mathrm{d}r)}
	\lesssim
		\|\psi\|_{{\frac{4}{6-q}}}\|f\|_2.
\end{align}
\end{lemma}
\begin{proof}
	The inequality follows from interpolation between the unitary case $q=2$ and the Strichartz inequality in one space dimension for $q=6$, that is,
	\begin{align*}
		\iint_{\R^2} |T_r f(x)|^6\,\mathrm{d}x	\,\mathrm{d}r \leq 12^{-\frac{1}{2}} \|f\|_2^6.
	\end{align*}
	For more details, see \cite[Lemma 2.1]{CHL15}.
\end{proof}

\begin{lemma}[Lemma 4.7 in \cite{CHL15}] \label{lem:continuity} $ $\\[-3ex]
\begin{enumerate}[leftmargin=10ex]
	\item[$\mathbf{d_{\mathrm{av}}=0:}$] If $2\leq \gamma_1\leq \gamma_2\leq 6$ and $\psi\in L^{\frac{4}{6-\gamma_2}}$ then the nonlinear nonlocal functional $N: L^2(\R)\to \R$ given by
		\begin{align*}
			L^2(\R)\ni f\mapsto N(f)=
			\iint_{\R^2}V(|T_r f|)\, \mathrm{d}x\, \psi \mathrm{d}r
		\end{align*}
		is locally Lipshitz continuous on $L^2$ in the sense that
		\begin{align*}
			|N(f_1)-N(f_2)| \lesssim \left(1+ \|f_1\|^{\gamma_2-1}_2 + \|f_2\|^{\gamma_2-1}_2 \right) \|f_1-f_2\|_2
		\end{align*}
		where the implicit constant depends only on the $L^{\frac{4}{6-\gamma_2}}$ norm of $\psi$.
	\item[$\mathbf{d_{\mathrm{av}}>0:}$] If $2\leq \gamma_1\leq \gamma_2<\infty$ and $\psi\in L^1$, then the nonlinear nonlocal functional $N: H^1(\R)\to \R$ given by
		\begin{align*}
			H^1(\R)\ni f\mapsto N(f)=
			\iint_{\R^2}V(|T_r f|)\, \mathrm{d}x\, \psi \mathrm{d}r
		\end{align*}
		is locally Lipschitz continuous in the sense that
		\begin{align*}
			|N(f_1)-N(f_2)| \lesssim \left(1+ \|f_1\|_{H^1}^{\gamma_2-2} + \|f_2\|_{H^1}^{\gamma_2-2}\right)\left( \|f_1\|_2+ \|f_2\|_2 \right) \|f_1-f_2\|_2.
		\end{align*}
	\end{enumerate}
\end{lemma}

The directional derivatives of the nonlinearity are given by

\begin{lemma} \label{lem:differentiability}
If $2\leq \gamma_1\leq\gamma_2\leq 6$ and $\psi\in L^1\cap L^{\frac{4}{6-\gamma_2}}$, respectively if $2\leq\gamma_1\leq\gamma_2<\infty$ and $\psi\in L^1$, then for any $f,h \in L^2(\R)$, respectively $f,h\in H^1(\R)$, the functional $N$ as above has directional derivative
given by
 \begin{align}\label{eq:Nderivative}
  D_h N(f)=
 	   \int_{\R} \Re \left\langle V'(|T_r f|)\tfrac{T_r f}{|T_r f|}, T_r h \right\rangle \,\psi \mathrm{d}r.
 \end{align}
 In particular, $h \mapsto D_h N(f)$ is real linear and continuous.
\end{lemma}

\begin{proof}
Let $f \in L^2(\R)$ and $t \neq 0$. For any $h\in L^2(\R)$ the difference quotient of $N$ is
\begin{align}\label{eq:derivative}
\frac{N(f+t h)-N(f)}{t}&=\frac{1}{t}\left[\iint_{\R^2}V(|T_r(f+t h)|)-V(|T_r f|)\,\mathrm{d}x \,\psi \mathrm{d}r\right]\notag\\
& =\frac{1}{t} \iint_{\R^2}\int _0^1 \frac{\mathrm{d}}{\mathrm{d}s} V(|T_r(f+s t h)|) \,\mathrm{d}s \,\mathrm{d}x \,\psi \mathrm{d}r.
\end{align}
Since $V$ is differentiable, we obtain
\begin{align*}
\frac{\mathrm{d}}{\mathrm{d}s} V(|T_r(f+s t h)|)=V ' (|T_r (f+s t h )|) \frac{t (T_rf \overline{T_rh}+T_rh \overline{T_rf}+2s t |T_r h|^2)}{2|T_r (f+s t h)|}
\end{align*}
and thus
\begin{align*}
\eqref{eq:derivative}=\iint_{\R^2} \int _0^1 V ' (|T_r (f+s t h )|) \frac{T_rf \overline{T_rh}+T_rh \overline{T_rf}+2s t |T_r h|^2}{2|T_r (f+s t h)|} \,\mathrm{d}s \,\mathrm{d}x \,\psi \mathrm{d}r.
\end{align*}
Under the assumptions $2\le \gamma_1\leq \gamma_2\leq 6$, respectively $2\leq \gamma_1\leq \gamma_2<\infty$, on the nonlinearity, Lebesgue's dominated convergence theorem, together with the continuity of $V'$, implies that for $t\to 0$,
\begin{align*}
D_hN(f)&=\iint_{\R^2} \int _0^1 V ' (|T_r f|) \frac{\Re (T_rf \overline{T_rh})}{|T_r f|} \,\mathrm{d}s \,\mathrm{d}x \,\psi \mathrm{d}r \\
&=\iint_{\R^2} V ' (|T_r f|) \frac{\Re (T_rf \overline{T_rh})}{|T_r f|}  \,\mathrm{d}x \,\psi \mathrm{d}r,
\end{align*}
which completes the proof of \eqref{eq:Nderivative}.
Linearity of the map $h\mapsto D_h N(f)$ is immediate from \eqref{eq:Nderivative}, to see the continuity observe that by assumption \ref{ass:growth},
\begin{align*}
	|D_h N(f)| &\leq \iint_{\R^2} | V'(|T_rf(x)|) | \, |T_r h(x)| \,\mathrm{d}x\,\psi\,\mathrm{d}r \\
	&\leq \iint_{\R^2} \left[ |T_r f(x)|^{\gamma_1-1} + |T_r f(x)|^{\gamma_2-1} \right] \, |T_r h(x)| \,\mathrm{d}x\,\psi\,\mathrm{d}r
\end{align*}
For $2 \leq \gamma \leq 6$, H\"older's inequality (with exponents $\frac{\gamma}{\gamma-1}$ and $\gamma$), implies the bound
\begin{align*}
	\iint_{\R^2} | T_r f(x) |^{\gamma-1} \, |T_r h(x)| \,\mathrm{d}x\,\psi\,\mathrm{d}r
	&\leq \|T_r f\|_{L^{\gamma}(\mathrm{d}x\,\psi\mathrm{d}r)}^{\gamma-1} \, \|T_r h\|_{L^{\gamma}(\mathrm{d}x\,\psi\mathrm{d}r)} \\
	&\lesssim \|f\|_2^{\gamma-1} \, \|h\|_2
\end{align*}
by Lemma \ref{lem:st-boundedness}. By linearity, this already shows continuity of $h\mapsto D_h N(f)$ in the case $d_{\mathrm{av}}=0$.

In the case of positive average dispersion, $d_{\mathrm{av}}>0$, we can use $f\in H^1$ and Cauchy-Schwarz to bound
\begin{align*}
	\iint_{\R^2} | T_r f(x) |^{\gamma-1} \, |T_r h(x)| \,\mathrm{d}x\,\psi\,\mathrm{d}r
	&\leq \sup_{r} \|T_r f\|_{\infty}^{\gamma-2} \, \|T_r f\|_{L^2(\mathrm{d}x\,\psi\mathrm{d}r)} \, \|T_r h\|_{L^{2}(\mathrm{d}x\,\psi\mathrm{d}r)} \\
	&\lesssim \|f\|_{H^1}^{\gamma-2} \, \|f\|_2 \, \|h\|_2,
\end{align*}
for $2\leq \gamma<\infty$, since
\begin{align*}
	\sup_{r\in\R} \|T_r f\|_{\infty} \leq \sup_{r\in\R} \|T_r f\|_{H^1} = \|f\|_{H^1}
\end{align*}
by the simple estimate $\|g\|_{\infty} \leq \left(\|g\|_2 \|g'\|_2\right)^{1/2} \leq \|g\|_{H^1}$ and unitarity of the free Schr\"odinger evolution $T_r$ on $H^1$.
\end{proof}

\begin{remark}\label{rem:duality}
	In the setting of $d_{\mathrm{av}} = 0$, that is, when working in $L^2(\R)$, Lemma \ref{lem:differentiability} also identifies the unique Riesz representative $\nabla N(f)$ (with respect to the {real} inner product $\mathrm{Re}\langle\cdot, \cdot\rangle$) of the continuous linear functional $h\mapsto D_h N(f)$ for fixed $f\in L^2(\R)$,
	\begin{align*}
	\mathrm{Re} \langle \nabla N(f), h\rangle = D_h N(f) = \Re \left\langle \int_{\R} T_r^{-1} \left[ V'(|T_r f|) \tfrac{T_r f}{|T_r f|} \right] \psi\,\mathrm{d}r, h \right\rangle,
 	\end{align*}
 	so
 	\begin{align*}
 		\nabla N(f) = \int_{\R} T_r^{-1} \left[ V'(|T_r f|) \tfrac{T_r f}{|T_r f|} \right] \psi\,\mathrm{d}r.
 	\end{align*}
\end{remark}

Even though we do not need the following for our main results, we state and prove
\begin{proposition}\label{prop:diff-nonlin}
	Assume that \eqref{eq:lipschitz} holds in addition to the assumptions of Lemma \ref{lem:differentiability}, with $2\leq \gamma \leq 6$, respectively, $2\leq\gamma<\infty$. Then the functional $H: \mathcal{H}^{d_\mathrm{av}} \to \R$,
	is of class $\mathcal{C}^1(\mathcal{H}^{d_\mathrm{av}}, \R)$.
\end{proposition}

\begin{proof}
Since $f\mapsto \|f'\|_2^2$ is a $\mathcal{C}^1$ functional on $H^1$, and the directional derivatives of $N$ are real linear, see Lemma \ref{lem:differentiability}, it suffices to show that $f \mapsto D_hN(f)$ is continuous for each $h\in \mathcal{H}^{d_\mathrm{av}}$. We start by estimating
\begin{align*}
	&\left|D_hN(f + g) - D_hN(f)\right| \\
	&\leq \iint_{\R^2} |T_r h| \left| V'(|T_r f + T_r g|) \frac{T_r f + T_r g}{|T_r f + T_r g|} - V'(|T_r f|) \frac{T_r f}{|T_r f|} \right| \,\mathrm{d}x \, \psi \mathrm{d}r.
\end{align*}

Observe that by assumption \ref{ass:growth} and inequality \eqref{eq:lipschitz}, for any $z,w \in \C$,
\begin{align*}
	&\left| V'(|z+w|) \frac{z+w}{|z+w|} - V'(|z|) \frac{z}{|z|} \right| \\
	&\leq \left| V'(|z+w|) - V'(|z|) \right| + \frac{|V'(|z+w|)|}{|z+w|} |w| + \left| V'(|z+w|) \right| \left| \frac{z}{|z+w|} - \frac{z}{|z|} \right| \\
	&= \left| V'(|z+w|) - V'(|z|) \right| +  \frac{|V'(|z+w|)|}{|z+w|} |w| +  \frac{|V'(|z+w|)|}{|z+w|} \left| |z| - |z+w| \right| \\
	&\lesssim |w| \left[ (|z| + |w| )^{\gamma-2} + (|z| + |w| )^{\gamma_1-2} + (|z| + |w| )^{\gamma_2-2} \right].
\end{align*}	
It follows that $\left|D_hN(f + g) - D_hN(f)\right|$ can be bounded by a sum of terms of the form
\begin{align*}
	\iint_{\R^2} |T_r h| |T_r g| \left( |T_r f| + |T_r g|\right)^{\gamma-2}\,\mathrm{d}x\,\psi\,\mathrm{d}r.
\end{align*}
Using H\"{o}lder's inequality, with exponents $\gamma$, $\gamma$, $\frac{\gamma}{\gamma-2}$, and Lemma \ref{lem:st-boundedness}, we obtain the bound
\begin{align*}
\begin{split}
	&\left|D_hN(f + g) - D_hN(f)\right|  \\
	&\lesssim \|h\|_2 \|g\|_2 \left[ \left( \|f\|_2 + \|g\|_2 \right)^{\gamma-2} + \left( \|f\|_2 + \|g\|_2 \right)^{\gamma_1-2} + \left( \|f\|_2 + \|g\|_2 \right)^{\gamma_2-2} \right],
\end{split}
\end{align*}
as in the proof of Lemma \ref{lem:differentiability},
which shows that all directional derivatives $D_hN$ are locally Lipshitz for each fixed $h\in \mathcal{H}^{d_\mathrm{av}}$.
Therefore, $H \in \mathcal{C}^1(\mathcal{H}^{d_\mathrm{av}};\R)$ if $\gamma \in [2,6]$ for $d_{\mathrm{av}}=0$. Similarly, one proves the $d_{\mathrm{av}}>0$ case with
$\gamma\geq 2$.
\end{proof}

\section{Existence of minimizers}

\subsection{Strict sub-additivity of the energy}

The crucial ingredient in establishing existence of minimizers is restoring (pre-)compactness of minimizing sequences modulo the natural symmetries of the problem. In this section we prove sub-additivity of the ground state energy with respect to $\lambda>0$.

While strict sub-additivity was established under the Ambrosetti-Rabinowitz condition \eqref{eq:AR} in \cite{CHL15}, in general, it  fails in the saturation regime, where $\frac{V'(a) a}{V(a)} \to 2$.
For any $C>0$, we define the quantity
\begin{align}\label{eq:Min-modified}
	E_{\lambda}^{d_{\mathrm{av}}}(C) \coloneq \inf\big\{H(f): f\in \mathcal{S}_{\lambda}^{d_{\mathrm{av}}}, \sup_{r\in\mathrm{supp}\,\psi}\|T_r f\|_{\infty} \leq C \big\}.
\end{align}

The following proposition says that strict sub-additivity still holds in the case of saturated nonlinearities, at least if minimizing sequences do not reach the saturation regime:

\begin{proposition}[Strict sub-additivity] \label{prop:strictsubadd}
Assume that \emph{\ref{ass:growth}} and \emph{\ref{ass:saturation}} hold, and that for any $\lambda> 0$ there exists a $C>0$ such that
\begin{align}
	E_{\lambda}^{d_{\mathrm{av}}} = E_{\lambda}^{d_{\mathrm{av}}}(C).
\end{align}
Then for any  $0<\delta<\frac{\lambda}{2}$, and $\lambda_1,\lambda_2 \geq \delta$ with $\lambda_1 + \lambda_2 \leq \lambda$, one has
\begin{align*}
	E_{\lambda_1}^{d_{\mathrm{av}}} + E_{\lambda_2}^{d_{\mathrm{av}}} \geq \left[1 - \left(2^{\frac{\kappa^*(C)}{2}} - 2\right) \left(\frac{\delta}{\lambda}\right)^{\frac{\kappa^*(C)}{2}} \right] E_{\lambda}^{d_{\mathrm{av}}},
\end{align*}
whenever $E_{\lambda}^{d_{\mathrm{av}}} \leq 0$, where $\kappa^*(C)\coloneq \inf_{0<a\leq C} \kappa(a) >2$.
\end{proposition}

\begin{remark}\label{rem:negativity}
	We will show in Propositions \ref{prop:wellposedness-zero} and \ref{prop:wellposedness-pos} that in fact for any $\lambda>0$, the ground state energy $E_{\lambda} \leq 0$.
	Proposition \ref{prop:strictsubadd} implies that
	\begin{align*}
		E_{\lambda_1}^{d_{\mathrm{av}}} + E_{\lambda_2}^{d_{\mathrm{av}}} > E_{\lambda_1+\lambda_2}^{d_{\mathrm{av}}}
	\end{align*}
	whenever $E_{\lambda_1+\lambda_2}^{d_{\mathrm{av}}} < 0$, i.e. $E_{\lambda}^{d_{\mathrm{av}}}$ is \emph{strictly sub-additive} if the ground state energy is \emph{strictly negative}.
	
	As shown in \cite{CHL15}, the strict sub-additivity of the ground state energy prevents minimizing sequences from splitting, in particular, minimizing sequences can be shown to be tight modulo the natural symmetries of the problem (shifts for $d_{\mathrm{av}} >0$ or shifts and boosts for $d_{\mathrm{av}}=0$).
	
\end{remark}

\begin{proof}[Proof of Proposition \ref{prop:strictsubadd}]
Set
\begin{align*}
	\chi(a) \coloneq \exp\left(-\int_{a_0}^a \frac{\kappa(b)}{b} \,\mathrm{d}b \right)
\end{align*}
for some $0<a_0\leq a$. Then $\chi(a_0) = 1$, $\chi'(a) = -\frac{\kappa(a)}{a} \chi(a)$, and therefore
\begin{align}\label{eq:A2bound}
	\chi(a) V(a) - V(a_0) \geq 0,
\end{align}
since $(\chi V)'\geq 0$ by assumption \ref{ass:saturation}. Setting $a_0 = sa$ for some $s\in (0,1]$, we obtain
\begin{align*}
	V(sa) &\leq \exp\left(- \int_{sa}^a \frac{\kappa(b)}{b}\,\mathrm{d}b \right) V(a) = \exp\left(- \int_{s}^1 \frac{\kappa(a b)}{b}\,\mathrm{d}b \right) V(a)\\
	&\leq \exp\left( - \inf_{\beta\in(0,1]} \kappa(a\beta) \int_s^1 \frac{\mathrm{d}b}{b} \right) V(a) = s^{\kappa^*(a)} V(a).
\end{align*}
Using that, by assumption \ref{ass:saturation},
\begin{align*}
	\inf_{0 < b \leq a} \kappa(b) \geq \inf_{0 < b \leq A} \kappa(b) = \kappa^*(A) >2
\end{align*}
for any finite $A\geq a>0$, we get
\begin{align*}
	V(s a) \leq s^{\kappa^*(A)} V(a), \quad \text{for all} \quad s\in(0,1], \, 0<a\leq A.
\end{align*}

At this point the $L^{\infty}$ bound comes into play, which guarantees that we always stay in a regime where saturation is not reached, that is, $\kappa^*>2$!
	
Indeed, since $|T_r f(x)| \leq \|T_r f\|_{\infty}\leq C$ for almost all $r\in\mathrm{supp}\,\psi$, we get for $0<\mu\leq 1$,
	\begin{align*}
		N(\mu^{1/2} f) = \iint_{\R^2} V(\mu^{1/2} |T_r f(x)|)\,\mathrm{d}x\,\psi(r)\,\mathrm{d}r \leq \mu^{\kappa^*(C)/2} N(f),
	\end{align*}
	and thus
	\begin{align*}
		E_{\mu \lambda}^{d_{\mathrm{av}}} &= \inf_{\|f\|_2^2 = \mu \lambda} \left( \frac{d_{\mathrm{av}}}{2} \|f'\|_2^2 - N(f) \right) \geq \inf_{\|g\|_2^2 = \lambda} \left(\mu \frac{d_{\mathrm{av}}}{2} \|g'\|_2^2 - \mu^{\kappa^*(C)/2} N(g) \right) \\
		&\geq \mu^{\kappa^*(C)/2} E_{\lambda}^{d_{\mathrm{av}}}.
	\end{align*}

As in \cite[Proposition 3.3]{CHL15}, we can now take $\lambda_j = \mu_j \lambda$, $j=1,2$, with $\mu_1 + \mu_2 \leq 1$, $\mu_1, \mu_2 \geq \frac{\delta}{\lambda}$. It then follows that
\begin{align*}
	E_{\lambda_1}^{d_{\mathrm{av}}} + E_{\lambda_2}^{d_{\mathrm{av}}} = E_{\mu_1 \lambda}^{d_{\mathrm{av}}} + E_{\mu_2 \lambda}^{d_{\mathrm{av}}} \geq \left(\mu_1^{\kappa^*(C)/2} + \mu_2^{\kappa^*(C)/2}\right) E_{\lambda}^{d_{\mathrm{av}}},
\end{align*}
and, since the function $t\mapsto (1+t)^{\kappa^*(C)/2} - 1 - t^{\kappa^*(C)/2}$ is increasing on $[1,\infty)$, we have
\begin{align*}
	\mu_1^{\kappa^*(C)/2} + \mu_2^{\kappa^*(C)/2} \leq 1 - \left(2^{\frac{\kappa^*(C)}{2}} - 2\right) \left(\frac{\delta}{\lambda}\right)^{\frac{\kappa^*(C)}{2}} <1
\end{align*}
for $\delta >0$ and $\kappa^*(C)>2$.

Now, if $E_{\lambda}^{d_{\mathrm{av}}} \leq 0$, the sub-additivity
\begin{align*}
	E_{\lambda_1}^{d_{\mathrm{av}}} + E_{\lambda_2}^{d_{\mathrm{av}}} \geq \left[1 - \left(2^{\frac{\kappa^*(C)}{2}} - 2\right) \left(\frac{\delta}{\lambda}\right)^{\frac{\kappa^*(C)}{2}} \right] E_{\lambda}^{d_{\mathrm{av}}},
\end{align*}
follows.
\end{proof}

\subsection{Thresholds}

It turns out that under the assumptions \ref{ass:growth}--\ref{ass:strictpos} on the nonlinear potential $V$, minimizers for $E_{\lambda}^{d_{\mathrm{av}}}$ may only exist for large enough $\lambda$. This is due to the fact that minimizing sequences can be shown to be pre-compact modulo translations, respectively,  translations and modulations, \emph{if} the energy is \emph{strictly negative}. The reason for this is sub-additivity: the ground state energy is \emph{strictly} sub-additive only if $E_{\lambda}^{d_{\mathrm{av}}}<0$!

This motivates
\begin{definition}[Threshold]\label{def:threshold}
	\begin{align*}
		\lambda_{\mathrm{cr}}^{d_{\mathrm{av}}} \coloneq \inf\{ \lambda>0: E_{\lambda}^{d_{\mathrm{av}}} < 0 \}.
	\end{align*}
\end{definition}

Assume that $E_{\lambda}^{d_{\mathrm{av}}}\leq 0$ for all $\lambda>0$ and $d_{\mathrm{av}}\geq 0$ (see Remark \ref{rem:negativity} about the validity of this assumption).
By the sub-additivity of the ground state energy, it immediately follows that
\begin{align*}
	E_{\lambda_1}^{d_{\mathrm{av}}} \geq E_{\lambda_1}^{d_{\mathrm{av}}} + E_{\lambda_2}^{d_{\mathrm{av}}} \geq E_{\lambda_1+\lambda_2}^{d_{\mathrm{av}}},
\end{align*}
where the latter inequality is strict whenever $E_{\lambda_1+\lambda_2}^{d_{\mathrm{av}}}<0$. In particular, the map $0<\lambda\mapsto E_{\lambda}^{d_{\mathrm{av}}}$ is decreasing and strictly decreasing where $E_{\lambda}^{d_{\mathrm{av}}}<0$.

Thus, $E_{\lambda}^{d_{\mathrm{av}}} = 0$  if $0<\lambda<\lambda_{\mathrm{cr}}^{d_{\mathrm{av}}}$ and $E_{\lambda}^{d_{\mathrm{av}}}<0$ if $\lambda>\lambda_{\mathrm{cr}}^{d_{\mathrm{av}}}$.

\begin{lemma}\label{lem:threshold}
	If $V$ satisfies assumptions \ref{ass:saturation} and \ref{ass:strictpos}, then $\lambda_{\mathrm{cr}}^{d_{\mathrm{av}}} < \infty$.
\end{lemma}

\begin{proof}
$\lambda_{\mathrm{cr}}^{d_{\mathrm{av}}} < \infty$ if and only if $E_{\lambda}^{d_{\mathrm{av}}} <0$ for some $\lambda>0$. The claim therefore follows if we can find a suitable trial function with negative energy $H$, at least for large enough $\lambda>0$.

Observe that by \ref{ass:saturation}, we again have the bound \eqref{eq:A2bound} on $V$. Let $a_*>0$ be such that $V(a_*)>0$, which exists by \ref{ass:strictpos}. Then
\begin{align*}
	V(a) \geq \exp\left(\int_{a_*}^a \frac{\kappa(b)}{b}\,\mathrm{d}b \right) V(a_*)\, \1_{[a_*, \infty)}(a),
\end{align*}
where for $0<a<a_*$ we just used the fact that $V(a) \geq 0$. Since by \ref{ass:saturation}, $\inf_{b>0} \kappa(b) \geq 2$, we get the lower bound
\begin{align}\label{eq:Vlowerbound}
	V(a) \geq \left(\frac{a}{a_*}\right)^2 V(a_*) \, \1_{[a_*, \infty)}(a).
\end{align}

Consider now centered Gaussian test functions
 	\begin{align}\label{eq:gaussian}
 		g_{\sigma_{0}}(x) = A_0 \,
 				\mathrm{e}^{-\frac{x^2}{\sigma_0}}, \quad \sigma_0>0,
 	\end{align}
where $A_0 = \left(\frac{2\lambda^2}{\pi \sigma_0}\right)^{1/4}$ is chosen such that $\|g_{\sigma_0}\|_2^2 = \lambda$. Then $\|g'_{\sigma_0}\|_2^2 = \frac{\lambda}{\sigma_0}$ and the time evolution is given by
\begin{align}\label{eq:evolvedgaussian}
  	T_rg_{\sigma_0}(x) =  A_0 \left( \frac{\sigma_0}{\sigma(r)} \right)^{1/2} \mathrm{e}^{-\frac{x^2}{\sigma(r)}}, \quad \sigma(r)= \sigma_0+4\mathrm{i}r,
\end{align}
thus,
\begin{align*}
	|T_rg_{\sigma_0}(x)| =  A_0 \left( \frac{\sigma_0^2}{\sigma_0^2 + (4r)^2} \right)^{1/4} \mathrm{e}^{-\frac{\sigma_0 x^2}{\sigma_0^2 + (4r)^2}}.
\end{align*}
We therefore have $|T_rg_{\sigma_0}(x)| \leq A_0$ for all $x\in\R$ and $r\in \R$. 
If $|x|\leq \sqrt{\sigma_0}$, we also have the lower bound
\begin{align*}
	|T_rg_{\sigma_0}(x)| \geq A_0 \left( \frac{\sigma_0^2}{\sigma_0^2 + (4r)^2} \right)^{1/4} \mathrm{e}^{-\frac{\sigma_0^2}{\sigma_0^2 + (4r)^2}},
\end{align*}
hence choosing $R>0$ such that $\mathrm{supp}\,\psi \subset [-R, R]$, we have 
\begin{align*}
	\frac{A_0}{2} \leq |T_rg_{\sigma_0}(x)| \leq A_0
\end{align*}
for all $|x| \leq \sqrt{\sigma_0}$ and all $|r|\leq R$, assuming $\sigma_0>4R$.

Now set $\sigma_0 = \lambda$ for $\lambda$ large enough. Then $\|g_{\lambda}'\|_2 = 1$ and $A_0 = \left(\frac{2\lambda}{\pi}\right)^{1/4}$.
It follows with \eqref{eq:Vlowerbound} that
\begin{align*}
	\int_{\R} V(|T_rg_{\lambda}(x)|)\,\mathrm{d}x
	&= \int_{|x|\leq \sqrt{\lambda}} V(|T_rg_{\lambda}(x)|)\,\mathrm{d}x + \int_{|x|>\sqrt{\lambda}} V(|T_rg_{\lambda}(x)|)\,\mathrm{d}x \\
	&\geq \int_{|x|\leq \sqrt{\lambda}} \left(\frac{|T_r g_{\lambda}(x)|}{a_*}\right)^2 V(a_*)\, \1_{[a_*, \infty)}(|T_r g_{\lambda}(x)|)\, \mathrm{d}x \\
	&\geq 2\sqrt{\lambda} \left(\frac{A_0}{2a_*} \right)^2 V(a_*) \, \1_{[a_*, \infty)}\big(\tfrac{A_0}{2}\big),
\end{align*}
so for $\lambda$ large enough, since $A_0 \sim \lambda^{1/4}$,
\begin{align*}
	N(g_{\lambda}) = \iint_{\R^2} V(|T_rg_{\lambda}(x)|)\,\mathrm{d}x\,\psi\mathrm{d}r \gtrsim \lambda 
\end{align*}
and the energy is bounded by
\begin{align*}
	H(g_{\lambda}) &= \frac{d_\mathrm{av}}{2} \|g_{\lambda}'\|_2^2 - N(g_{\lambda})
	\leq \frac{d_\mathrm{av}}{2} - C \lambda, 
\end{align*}
for some constant $C>0$. Thus, choosing $\lambda>0$ large enough, we can always achieve $H(g_{\lambda})<0$, so $E_{\lambda}^{d_{\mathrm{av}}} = \inf_{\|f\|_2^2 = \lambda} H(f) \leq H(g_{\lambda}) <0$.
\end{proof}

\begin{lemma}
	If $V$ satisfies assumptions \ref{ass:growth}, \ref{ass:saturation}, and  \ref{ass:zerothreshold}, then $\lambda_{\mathrm{cr}}^{d_{\mathrm{av}}} = 0$ for all $d_{\mathrm{av}}\ge 0$.
\end{lemma}

\begin{proof}
Let $\lambda>0$. We begin with $d_{\mathrm{av}}=0$, that is, assume that there exists $\epsilon>0$ such that $V(a)>0$ for all $0<a\leq \epsilon$. Let $g_{\sigma_0}$ be the centered Gaussian \eqref{eq:gaussian} with $\|g_{\sigma_0}\|_2^2 = \lambda$. Then, by \eqref{eq:evolvedgaussian},
\begin{align}
  	|T_rg_{\sigma_0}(x)| \leq A_0 = \left(\frac{2\lambda^2}{\pi \sigma_0}\right)^{1/4}
\end{align}
for all $x\in\R$ and $r\in \R $.
Choosing $\sigma_0$ large enough, we can make $|T_rg_{\sigma_0}(x)| \leq \epsilon$, which implies $H(g_{\sigma_0}) = -N(g_{\sigma_0}) <0$ by \ref{ass:zerothreshold}, so $E_{\lambda}^{d_{\mathrm{av}}} < 0$. Since $\lambda>0$ was arbitrary, it follows that $\lambda_{\mathrm{cr}}^0 = 0$.

For $d_{\mathrm{av}}>0$ assume that there exist $\epsilon>0$ and $2<\gamma_0<6$ such that $V(a)\gtrsim a^{\gamma_0}$ for all $0<a\leq \epsilon$. We consider the same centered Gaussian $g_{\sigma_0}$ as above, with $\sigma_0$ so large that $|T_rg_{\sigma_0}(x)| \leq \epsilon$. It follows that
\begin{align*}
	N(g_{\sigma_0}) &= \iint_{\R^2} V(|T_r g_{\sigma_0}(x)|) \,\mathrm{d}x\,\psi\,\mathrm{d}r \gtrsim \iint_{\R^2} |T_r g_{\sigma_0}(x)|^{\gamma_0}\,\mathrm{d}x\,\psi\mathrm{d}r \\
	&= \left( \frac{\pi}{\gamma_0} \right)^{1/2} \left( \frac{2\lambda^2}{\pi} \right)^{\gamma_0/4} \sigma_0^{\frac{2-\gamma_0}{4}} \int_{\R} \frac{\psi(r)}{\left[ 1 + (4r/\sigma_0)^2 \right]^{\frac{\gamma_0-2}{4}}} \,\mathrm{d}r.
\end{align*}
Since $\|g'_{\sigma_0}\|_2^2 = \frac{\lambda}{\sigma_0}$, the energy of the Gaussian $g_{\sigma_0}$ is bounded by
\begin{align*}
	H(g_{\sigma_0}) \leq 
\frac{d_{\mathrm{av}} \lambda}{2\sigma_0} 
\left[ 1- \frac{C}{d_{\mathrm{av}} \lambda} \left( \frac{\pi}{\gamma_0} \right)^{1/2} \left( \frac{2\lambda^2}{\pi} \right)^{\gamma_0/4} \sigma_0^{\frac{6-\gamma_0}{4}} \int_{\R} \frac{\psi(r)}{\left[ 1 + (4r/\sigma_0)^2 \right]^{\frac{\gamma_0-2}{4}}} \,\mathrm{d}r \right]
\end{align*}
for some constant $C>0$. In particular, since $2<\gamma_0<6$ and
\begin{align*}
	\int_{\R} \frac{\psi(r)}{\left[ 1 + (4r/\sigma_0)^2 \right]^{\frac{\gamma_0-2}{4}}} \,\mathrm{d}r \to \|\psi\|_1 >0
\end{align*}
as $\sigma_0 \to \infty$ by Lebesgue's dominated convergence theorem, we can make $\sigma_0$ sufficiently large, such that $H(g_{\sigma_0}) < 0$. As $\lambda>0$ was arbitrary, this yields $\lambda_{\mathrm{cr}}^{d_{\mathrm{av}}} = 0$.
\end{proof}

The following quantity will be useful in proving the non-existence of minimizers in the positive average dispersion case for sub-critical $0<\lambda <\lambda_{\mathrm{cr}}^{d_{\mathrm{av}}}$. Fix $C>0$ and define
\begin{align*}
	R_C(\lambda):= \sup\left\{ \frac{N(\sqrt{\lambda} h)}{\lambda \|h'\|_2^2}: h\in H^1(\R)\setminus \{0\}, \|h\|_2 = 1, \|h'\|_2 \leq C \right\}.
\end{align*}

\begin{lemma}\label{lem:scalingR}
	Let $C>0$. If $V$ satisfies assumption \ref{ass:saturation}, then
	\begin{align*}
		R_C (\lambda) \geq \left( \frac{\lambda}{\lambda_0} \right)^{\frac{1}{2} \kappa^*\left(\sqrt{\lambda C}\right) - 1} R_C (\lambda_0)
	\end{align*}
	for all $\lambda\geq \lambda_0 > 0$, with $\kappa^*\left(\sqrt{\lambda C}\right) = \inf_{a\leq \sqrt{\lambda C}} \kappa(a)>2$.
\end{lemma}

This scaling property immediately implies

\begin{corollary}\label{cor:thresholds} Let $C>0$ and assume that $V$ obeys assumption \ref{ass:saturation}.
	If $\lambda_{\mathrm{cr}}^{d_{\mathrm{av}}} > 0$, then
	\begin{align*}
		R_C (\lambda) < \frac{d_{\mathrm{av}}}{2} \quad \text{for all} \quad 0<\lambda<\lambda_{\mathrm{cr}}^{d_{\mathrm{av}}}.
	\end{align*}
\end{corollary}
\begin{proof}
	Let $\lambda_{\mathrm{cr}}^{d_{\mathrm{av}}}>0$, and assume that there exists $0<\lambda_1<\lambda_{\mathrm{cr}}^{d_{\mathrm{av}}}$ such that $R_C(\lambda_1) \geq \frac{d_{\mathrm{av}}}{2}$. Pick $\lambda_2 \in (\lambda_1, \lambda_{\mathrm{cr}}^{d_{\mathrm{av}}})$, then Lemma \ref{lem:scalingR} implies
	\begin{align*}
		R_C(\lambda_2) \geq \left( \frac{\lambda_2}{\lambda_1} \right)^{\frac{1}{2} \kappa^*\left(\sqrt{\lambda_2 C}\right) - 1} R_C (\lambda_1) > R_C (\lambda_1) \geq \frac{d_{\mathrm{av}}}{2},
	\end{align*}
	since $\kappa^*\left(\sqrt{\lambda C}\right) >2$.
	In particular,
	\begin{align*}
		E_{\lambda_2}^{d_{\mathrm{av}}}
		&= \inf_{\|g\|_2=1} \left( \frac{d_{\mathrm{av}}}{2} \lambda_2 \|g'\|_2^2 - N(\sqrt{\lambda_2} g) \right)
		\leq \inf_{\substack{\|g\|_2=1 \\ \|g'\|_2 \leq C}} \left( \frac{d_{\mathrm{av}}}{2} \lambda_2 \|g'\|_2^2 - N(\sqrt{\lambda_2} g) \right) \\
		&= \inf_{\substack{\|g\|_2=1 \\ \|g'\|_2 \leq C}} \lambda_2 \|g'\|_2^2 \left( \frac{d_{\mathrm{av}}}{2}  - \frac{N(\sqrt{\lambda_2} g)}{\lambda_2 \|g'\|_2^2} \right)  \leq \lambda_2  C^2 \inf_{\substack{\|g\|_2=1 \\ \|g'\|_2 \leq C}} \left( \frac{d_{\mathrm{av}}}{2}  - \frac{N(\sqrt{\lambda_2} g)}{\lambda_2 \|g'\|_2^2} \right)   \\
		&= \lambda_2  C^2 \left(\frac{d_{\mathrm{av}}}{2} - \sup_{\substack{\|g\|_2=1 \\ \|g'\|_2 \leq C}} \frac{N(\sqrt{\lambda_2} g)}{\lambda_2 \|g'\|_2^2} \right) = \lambda_2  C^2 \left(\frac{d_{\mathrm{av}}}{2} - R_C(\lambda_2) \right) < 0,
	\end{align*}
	in contradiction to $\lambda_2 < \lambda_{\mathrm{cr}}^{d_{\mathrm{av}}}$ and the definition of $\lambda_{\mathrm{cr}}^{d_{\mathrm{av}}}$.
\end{proof}

\begin{proof}[Proof of Lemma \ref{lem:scalingR}]
	Let $h\in H^1\setminus \{0\}$ with $\|h\|_2 = 1$ and $\|h'\|_2 \leq C$, and define the function
	\begin{align*}
		A(s):= s^{-2} N(sh) 
	\end{align*}
	for $s>0$. 
	Then
	\begin{align*}
		A'(s) = s^{-3} \left( s D_h N(sh) - 2 N(sh) \right),
	\end{align*}
	and by assumption \ref{ass:saturation},
	\begin{align*}
		s D_h N(sh) - 2 N(sh) &= \iint_{\R^2} \Bigl[ V'(|T_r(sh)|) |T_r(sh)| - 2 V(|T_r(sh)|) \Bigr] \,\mathrm{d}x\,\psi\,\mathrm{d}r \\
		&\geq \iint_{\R^2} \Bigl[ \kappa(|T_r(sh)|) - 2 \Bigr] V(|T_r(sh)|)\,\mathrm{d}x\,\psi\,\mathrm{d}r.
	\end{align*}
	Since for any $f\in H^1(\R)$ the simple inequality $\|f\|_{\infty}^2 \leq \|f\|_2 \|f'\|_2$ holds, we get
	\begin{align*}
		\|T_r(sh)\|_{\infty} = s \|T_r h\|_{\infty} \leq s \|T_r h\|_2^{1/2} \|T_r h'\|_2^{1/2} = s \|h\|_2^{1/2} \|h'\|_2^{1/2} \leq s \sqrt{C},
	\end{align*}
	where we made use of the fact that $T_r$ commutes with differentiation and is unitary on $L^2$, as well as the properties of $h$.
	It follows that
	\begin{align*}
		s D_h N(sh) - 2 N(sh) &\geq \left(\inf_{a\leq s\sqrt{C}} \kappa(a)-2 \right) N(sh) \\
		&\geq \left(\inf_{a\leq t\sqrt{C}} \kappa(a)-2 \right) N(sh) = \left(\kappa^*\big(t\sqrt{C}\big) -2 \right) N(sh)
	\end{align*}
	for all $0<s \leq t$, $t>0$, and thus the function $A$ satisfies the differential inequality
	\begin{align*}
		A'(s) \geq \left(\kappa^*\big(t\sqrt{C}\big)-2 \right) \, s^{-1}\, A(s),
	\end{align*}
	which yields
	\begin{align*}
		A(t) \geq \left(\frac{t}{t_0}\right)^{\kappa^*\big(t\sqrt{C}\big)-2} \, A(t_0)
	\end{align*}
	for any $t\geq t_0 > 0$. In particular, we have
	\begin{align*}
		R_{C}(\lambda) = \sup_{\substack{\|h\|_2 = 1 \\ \|h'\|_2 \leq C}} \frac{N(\sqrt{\lambda} h)}{\lambda \|h'\|_2^2}
		&\geq \left(\frac{\lambda}{\lambda_0}\right)^{\frac{1}{2}\kappa^*\big(\sqrt{\lambda C}\big)-1} \sup_{\substack{\|h\|_2 = 1 \\ \|h'\|_2 \leq C}} \frac{N(\sqrt{\lambda_0} h)}{\lambda_0 \|h'\|_2^2} \\
		&= \left(\frac{\lambda}{\lambda_0}\right)^{\frac{1}{2}\kappa^*\big(\sqrt{\lambda C}\big)-1} R_C(\lambda_0)
	\end{align*}
	for all $\lambda \geq \lambda_0 > 0$.
\end{proof}

\subsection{Existence of minimizers for zero average dispersion}
We start by establishing the existence of minimizers in the singular case $d_{\mathrm{av}} = 0$. Throughout this section, we assume that \ref{ass:growth}, \ref{ass:saturation}, and \ref{ass:strictpos} hold with $3\le\gamma_1\leq\gamma_2< 5$, and that $\psi$ is compactly supported with $\psi \in L^{\frac{4}{5-\gamma_2}+\delta}$ for some $\delta>0$. This $L^p$ condition on $\psi$ ensures that the $L^p$ condition in \cite{CHL15} holds, in particular, all their multilinear estimates and splitting estimates continue to hold in our setting.

\begin{proposition}\label{prop:wellposedness-zero}
For any $\lambda>0$, the energy functional $H = -N$ is bounded below on $\mathcal{S}_{\lambda}^0$ and
\begin{align*}
		-\infty < E_{\lambda}^0 \leq 0.
\end{align*}	
\end{proposition}

\begin{proof}
Let $\lambda>0$. Integrating the bound on $V'$ in \ref{ass:growth} yields
\begin{align}\label{eq:Vbound}
	|V(a)| \lesssim a^{\gamma_1} + a^{\gamma_2}
\end{align}
	and therefore
\begin{align*}
	N(f) \lesssim \int_{\R} \|T_rf\|_{\gamma_1}^{\gamma_1}\,\psi\mathrm{d}r + \int_{\R} \|T_rf\|_{\gamma_2}^{\gamma_2}\,\psi\mathrm{d}r \lesssim \|f\|_2^{\gamma_1} + \|f\|_2^{\gamma_2}
\end{align*}
	by Lemma \ref{lem:st-boundedness}. It follows that
\begin{align*}
	E_{\lambda}^0 = \inf_{\|f\|_2^2 = \lambda} H(f) = - \sup_{\|f\|_2^2 = \lambda} N(f) \gtrsim - \left(\lambda^{\frac{\gamma_1}{2}} + \lambda^{\frac{\gamma_2}{2}} \right) >-\infty .
\end{align*}
Since $V(a) \geq 0$ for any $a>0$, clearly $H(f) = -N(f) \leq 0$ for any $f\in \mathcal{S}_{\lambda}^0$ and therefore $E_{\lambda}^0 \leq 0$.
\end{proof}

The following lemma is a generalization of a result by \textsc{Kunze} \cite[Lemma 2.12]{Kun04}, and establishes $L^{\infty}$ bounds on the time evolved gradient of $H$.
\begin{lemma}\label{lem:kunzebound}
	Let $f \in L^2(\R)$, $3\leq \gamma_1 \leq \gamma_2 < 5$, $\psi\in L^{\frac{4}{5-\gamma_2}+\delta}$ for some $\delta>0$, and  $\psi$ compactly supported.
	Then $T_s \nabla H(f) \in L^{\infty}(\R)$ and
	\begin{align}\label{eq:kunzebound}
		\sup_{s\in\R} \|T_s \nabla H(f)\|_{\infty} \lesssim \|f\|_2^{\gamma_1-1} + \|f\|_2^{\gamma_2-1},
	\end{align}
	where the implicit constant depends on $\|\psi\|_{{\frac{4}{5-\gamma_2}}+\delta}$.
\end{lemma}

\begin{proof}
	We have
	\begin{align*}
		&\|T_s \nabla H(f) \|_{\infty} = \sup_{\|g\|_{1} = 1} \left| \Re \langle T_s \nabla H(f), g \rangle\right| = \sup_{\|g\|_{1} = 1} \left| \Re \langle \nabla H(f), T_{-s} g \rangle\right| \\
		&= \sup_{\|g\|_{1} = 1} \left|\Re \int_{\R} \left\langle V'(|T_r f|) \frac{T_r f}{|T_r f|}, T_{r-s} g \right\rangle \,\psi(r)\,\mathrm{d}r \right| .
	\end{align*}
	Using the basic dispersive estimate for the free Schr\"odinger evolution, $\|T_s g\|_{\infty} \lesssim |s|^{-1/2} \|g\|_{1}$ for all $s\neq 0$, we obtain, together with assumption \ref{ass:growth},
	\begin{align}\label{eq:uniform}
		\|T_s \nabla H(f) \|_{\infty} &\lesssim \int_{\R} \frac{\psi(r)}{|r-s|^{1/2}} \int_{\R} |V'(|T_r f|)|\,\mathrm{d}x\,\mathrm{d}r\nonumber\\
		&\lesssim \int_{\R} \frac{\psi(r)}{|r-s|^{1/2}} \left( \|T_r f\|_{{\gamma_1-1}}^{\gamma_1-1} + \|T_r f\|_{{\gamma_2-1}}^{\gamma_2-1} \right) \,\mathrm{d}r.
	\end{align}
	An application of H\"older's inequality then yields
	\begin{align*}
		\int_{\R} \frac{\psi(r)}{|r-s|^{1/2}} \|T_r f\|_{\gamma-1}^{\gamma-1} \,\mathrm{d}r
		\leq \||\cdot -s|^{-1/2}\psi\|_{{\frac{p}{p-1}}} \, \left( \int_{\R} \|T_r f\|_{{\gamma-1}}^{p(\gamma-1)} \,\mathrm{d}r \right)^{1/p} .
	\end{align*}
	The pair $(\gamma-1, p(\gamma-1))$ is Strichartz admissible if $\gamma-1 \geq 2$ and $\frac{2}{p(\gamma-1)} = \frac{1}{2} - \frac{1}{\gamma-1}$, that is, $p=\frac{4}{\gamma-3}$. Note that $p \geq 1$ if $\gamma \leq 7$. In this case,
	\begin{align*}
		\int_{\R} \|T_r f\|_{\gamma-1}^{p(\gamma-1)} \,\mathrm{d}r \lesssim \|f\|_2^{p(\gamma-1)}
	\end{align*}
	by Strichartz' inequality, and thus,
	\begin{align*}
		\int_{\R} \frac{\psi(r)}{|r-s|^{1/2}} \|T_r f\|_{{\gamma-1}}^{\gamma-1} \,\mathrm{d}r \lesssim \||\cdot -s|^{-1/2}\psi\|_{{\frac{4}{7-\gamma}}} \|f\|_2^{\gamma-1}.
	\end{align*}
	Setting $\alpha= \frac{2}{7-\gamma}$, we see that we have to bound $\int |r-s|^{-\alpha} \psi^{2\alpha}(r)\, \mathrm{d}r$ uniformly in $s$. Let $\theta>1$ and apply H\"older's inequality once more to see
	\begin{align*}
	  \int |r-s|^{-\alpha} \psi^{2\alpha}(r)\, \mathrm{d}r
	    \le \left( \int_{\mathrm{supp}\,\psi} |r-s|^{-\alpha\theta} \, \mathrm{d} r \right)^{\frac{1}{\theta}} \left( \int \psi(r)^{\frac{2\alpha\theta}{\theta-1}} \right)^{\frac{\theta-1}{\theta}} .
	\end{align*}
	As long as $\alpha\theta<1$, we have
	\begin{align*}
	  \sup_{s\in\R} \int_{\mathrm{supp}\,\psi} |r-s|^{-\alpha\theta} \, \mathrm{d} r <\infty
	\end{align*}
	since $\mathrm{supp}\,\psi$ is compact. So we need $\alpha<1/\theta$, which is equivalent to
	\begin{align*}
	  \frac{2\alpha\theta}{\theta-1} > \frac{2\alpha}{1-\alpha} = \frac{4}{5-\gamma}. 
	\end{align*}
	Since $\psi$ is compactly supported and  $\psi\in L^{\frac{4}{5-\gamma_2}+\delta}$ for some $\delta>0$, we see that, setting $\alpha_j = \frac{4}{7-\gamma_j}$, there exist $\theta_j>1$ with $\alpha_j\theta_j<1$ and
	$\frac{2\alpha_j\theta_j}{\theta_j-1}= \frac{4}{5-\gamma_j}+\delta$. This shows that both terms on the right hand side of \eqref{eq:uniform} can be bounded uniformly in $s\in\R $.
\end{proof}

\begin{lemma}\label{lem:modified}
	Assume that $3\leq \gamma_1 \leq \gamma_2 <5$ and that $\psi\in L^{\frac{4}{5-\gamma_2}+\delta}$ for some $\delta>0$. Let $(u_n)_{n\in\N}\subset L^2(\R)$, $\|u_n\|_2^2=\lambda$ for all $n\in\N$, be a minimizing sequence for $E_{\lambda}^0$. If $E_{\lambda}^0<0$, then there exists another minimizing sequence $(v_n)_{n\in\N}\subset L^2\cap L^{\infty}(\R)$ with
	\begin{align*}
		\sup_{r\in\R} \|T_r v_n\|_{\infty} \leq C_{\lambda}.
	\end{align*}
\end{lemma}

\begin{proof}
\textbf{Step 1 (Construction of a modified minimizing sequence).}
	Since $H$ satisfies all the requirements of Ekeland's variational principle (see Appendix \ref{app:ekeland}), there exists another minimizing sequence $(w_n)_{n\in\N}\subset \mathcal{S}_{\lambda}^0$, such that $H(w_n) \leq H(u_n)$ for all $n\in\N$, $\|w_n - u_n\|_2 \to 0$ as $n\to\infty$, and
	\begin{align}\label{eq:ekeland-modified}
		\nabla H(w_n) - \left\langle \nabla H(w_n), \frac{w_n}{\|w_n\|_2} \right\rangle \frac{w_n}{\|w_n\|_2} \to 0 \quad \text{as} \quad n\to \infty
	\end{align}
	strongly in $L^2$, where $\nabla H(f) = - \int_{\R} T_r^{-1}\left[V'(|T_r f|) \frac{T_r f}{|T_r f|} \right] \,\psi\mathrm{d}r$, see Remark \ref{rem:duality}. Write
	\begin{align}\label{eq:Ekeland-remainder}
		g_n := \nabla H(w_n) + \sigma_n \frac{w_n}{\|w_n\|_2} = \nabla H(w_n) + \sigma_n \frac{w_n}{\sqrt{\lambda}}, \quad n\in\N,
	\end{align}
	with $\sigma_n \coloneq -\left\langle \nabla H(w_n), \frac{w_n}{\|w_n\|_2} \right\rangle$. Then $g_n \to 0$ strongly in $L^2$ for $n\to\infty$ by \eqref{eq:ekeland-modified}.
	
	By assumption \ref{ass:saturation},
	\begin{align*}
		-\langle \nabla H(w_n), w_n \rangle &= D_{w_n}N(w_n)= \iint_{\R^2} V'(|T_r w_n|)\,|T_r w_n|\,\mathrm{d}x\,\psi(r)\,\mathrm{d}r\\
		&\geq 2 N(w_n) = -2 H(w_n) \stackrel{n\to\infty}{\longrightarrow} -2 E_{\lambda}^0 >0 ,
	\end{align*}
	so, picking a subsequence if necessary, we can assume that $\sigma_n \geq -\frac{E_{\lambda}^0}{\sqrt{\lambda}} >0$ for all $n\in\N$. Therefore, $\sigma_n^{-1}$ is uniformly bounded and $\sigma_n^{-1}g_n  \to 0$ as $n\to\infty$.

	Now define the sequence
	\begin{align*}
		v_n :=  - \sqrt{\lambda} \, \frac{\nabla H(w_n)}{\|\nabla H(w_n)\|_2}, \quad \|v_n\|_2^2 = \lambda, \quad n\in\N.
	\end{align*}
	We will show that $(v_n)_{n\in\N} \subset \mathcal{S}_{\lambda}^0$ is again a minimizing sequence for $H$. Indeed,
	\begin{align*}
		\|v_n - w_n\|_2
		&= \frac{\sqrt{\lambda}}{\sigma_n} \left\| \left( 1- \frac{\sigma_n}{\|\nabla H(w_n)\|_2}  \right) \nabla H(w_n) - g_n \right\|_2 \\
		&\leq \sqrt{\lambda} \left| 1- \frac{\sigma_n}{\|\nabla H(w_n)\|_2} \right| \frac{\|\nabla H(w_n)\|_2}{\sigma_n} + \sqrt{\lambda} \frac{\|g_n\|_2}{\sigma_n} .
	\end{align*}
	Since $\sigma_n^{-1} g_n \to 0$ in $L^2$, it remains to show that
	\begin{align*}
		\frac{\sigma_n}{\|\nabla H(w_n)\|_2} \to 1 
	\end{align*}
	as $n\to\infty$. But from \eqref{eq:Ekeland-remainder} we have
	\begin{align*}
		\|\nabla H(w_n) \|_2^2 = \|g_n\|_2^2 + \sigma_n^2 - 2 \sigma_n \Re \left\langle \frac{w_n}{\|w_n\|_2}, g_n \right\rangle,
	\end{align*}
	so
	\begin{align*}
	\frac{\|\nabla H(w_n)\|_2^2}{\sigma_n^2} = 1 &+ \frac{\|g_n\|_2^2}{\sigma_n^2} - 2 \, \Re \left\langle \frac{w_n}{\|w_n\|_2}, \frac{g_n}{\sigma_n} \right\rangle  \to 1
	\end{align*}
	as $n\to\infty$ since $\sigma_n^{-1} g_n \to 0$ in $L^2$.

\medskip
\noindent
\textbf{Step 2 ($L^{\infty}$ boundedness of the modified minimizing sequence).}
	By Lemma \ref{lem:kunzebound} and the bound
	\begin{align*}
		\|w_n\|_2\|\nabla H(w_n)\|_2 \geq \left|\Re\langle\nabla H(w_n), w_n \rangle \right| \geq -2 E_{\lambda}^0 > 0,
	\end{align*}
	we obtain
	\begin{align*}
		\|T_s v_n\|_{\infty} &= \frac{\sqrt{\lambda}}{\|\nabla H(w_n)\|_2} \|T_s \nabla H(w_n)\|_{\infty} \\
		&\lesssim \frac{\lambda}{|E^0_\lambda|}\left( \|w_n\|_2^{\gamma_1-1} + \|w_n\|_2^{\gamma_2-1}\right) = \left(\lambda^{\frac{\gamma_1+1}{2}} + \lambda^{\frac{\gamma_2+1}{2}}\right)/|E^0_\lambda|.
	\end{align*}
\end{proof}

We can now turn to the proof of existence of dispersion managed solitons in the case of zero average/ dispersion.

\begin{proof}[Proof of Theorem \ref{thm:zero}]
We start with $0<\lambda<\lambda_{\mathrm{cr}}^0$. Since by Proposition \ref{prop:wellposedness-zero} $E_{\lambda}^0 \leq 0$, the definition of the threshold (Definition \ref{def:threshold}) implies that $E_{\lambda}^0 = 0$, proving part (i) of the theorem.

Assume now that $\lambda>\lambda_{\mathrm{cr}}^0$. Then, by definition, $E_{\lambda}^0 < 0$.

Let $(v_n)_{n\in\N}\subset \mathcal{S}_{\lambda}^0\cap L^{\infty}(\R)$ be the minimizing sequence constructed in Lemma \ref{lem:modified}, such that $\|T_r v_n\|_{\infty} \leq C_{\lambda}$ for some uniform constant $C_{\lambda}$.

By Proposition \ref{prop:strictsubadd}, the ground state energy $E_{\lambda}^0$ is strictly sub-additive along $(v_n)_{n\in\N}$.
Once we have strict sub-additivity, the bound (4.7) from \cite[Proposition 4.4]{CHL15} again holds. Then one can use this, similarly to the proof of \cite[Proposition 4.6]{CHL15}, to show that the sequence $(v_n)_{n\in\N}$ is tight (that is, $|v_n(x)|^2\,\mathrm{d}x$ and $|\widehat{v}_n(\eta)|^2\,\mathrm{d}\eta$ are tight in the sense of measures) modulo shifts and boosts, i.e. there exist shifts $y_n$ and boosts $\xi_n$ such that
\begin{align*}
	\lim_{R\to\infty} \sup_{n\in\N} \int_{|x-y_n|>R} |v_n(x)|^2\,\mathrm{d}x &= 0, \\
    \lim_{L\to\infty} \sup_{n\in\N} \int_{|\eta-\xi_n|>L} |\widehat{v}_n(\eta)|^2\,\mathrm{d}\eta &= 0	.
\end{align*}
Let $f_n(x) \coloneq \mathrm{e}^{\I \xi_n x} v_n(x- y_n)$, $n\in\N$, be the shifted and boosted minimizing sequence. Then by the invariance of $H$ under shifts and boosts, $(f_n)_{n\in\N}$ is again a minimizing sequence with $\|f_n\|_2^2 = \|v_n\|_2^2 = \lambda$. Since $|f_n(x)| = |v_n(x-y_n)|$ and $|\widehat{f_n}(\eta)| = |\widehat{v}_n(\eta-\xi_n)|$, the sequence $(f_n)_{n\in\N}$ is also tight.

Since the sequence $(f_n)_{n\in\N}$ is bounded in $L^2(\R)$, there exists a weakly convergent subsequence (again denoted $(f_n)_{n\in\N}$) by the weak compactness of the unit ball. Since this subsequence is also tight, it converges even strongly in $L^2(\R)$ to some $f \in L^2$. By continuity of the $L^2$ norm and the nonlinearity $N$ under strong $L^2$-convergence, we have
\begin{align*}
	E_{\lambda}^0 \leq H(f) = -N(f) = \lim_{n\to\infty} -N(f_n) = E_{\lambda}^0,
\end{align*}
since $(f_n)_{n\in\N}$ is minimizing. Thus $f$ is a minimizer of the variational problem \eqref{eq:varprob} for $d_{\mathrm{av}} = 0$.

The Euler-Lagrange equation of the constrained minimization problem is the dispersion management equation \eqref{eq:DM} and it is a standard exercise to show that the minimizer $h$ found above is a weak solution of the Euler-Lagrange equation,
\begin{align}\label{eq:EL0}
	\omega \, \langle f, g \rangle = - D_g N(f) = -\int_{\R} \left\langle V'(|T_r f|) \frac{T_r f}{|T_r f|}, T_r g \right\rangle \,\psi\,\mathrm{d}r
\end{align}
for all $g\in L^2(\R)$, see also \cite{CHL15} for more details.
In particular, Lemma \ref{lem:kunzebound} implies that $T_s f \in L^{\infty}(\R)$ for almost all $s\in\mathrm{supp}\psi$. Inserting $g=f$ as test function in \eqref{eq:EL0} yields
\begin{align*}
	\omega \|f\|_2^2 = \omega \lambda &= - \iint_{\R^2} V'(|T_r f(x)|) \, |T_r f(x)|\,\mathrm{d}x\,\psi\,\mathrm{d}r \\
	&\leq -\iint_{\R^2} \kappa(|T_r f(x)|) \, V(|T_r f(x)|) \,\mathrm{d}x\,\psi\,\mathrm{d}r \\
	&\leq -\kappa^*(C_{\lambda}) N(f) < - 2 N(f) = 2 E_{\lambda}^0,
\end{align*}
by assumption \ref{ass:saturation} and the uniform bound on the minimizer, so $\omega < \frac{2 E_{\lambda}^0}{\lambda}$.
\end{proof}

\subsection{Existence of minimizers for positive average dispersion}

The situation is much easier in the positive average dispersion case, since the uniform $L^{\infty}$ bound is directly provided by the simple bound
\begin{align}\label{eq:1dsobolev}
	\|h\|_{\infty}^2 \leq \|h\|_2 \|h'\|_2 \leq \|h\|_{H^1}^2
\end{align}
for any $h\in H^1(\R)$, i.e., the Sobolev embedding $H^1(\R) \subset L^{\infty}(\R)$. We will assume throughout this section that assumptions \ref{ass:growth}, \ref{ass:saturation}, and \ref{ass:strictpos} hold with $2<\gamma_1 \leq \gamma_2<10$. We further assume that $\psi$ is compactly supported and $\psi\in L^{a_{\delta}}$ for some $\delta>0$, where $a_{\delta} \coloneq \max\{1, \frac{4}{10-\gamma_2}+\delta\}$.

\begin{proposition}\label{prop:wellposedness-pos}
The energy functional $H$ is bounded below on $\mathcal{S}_{\lambda}^{d_{\mathrm{av}}}$ for any $\lambda>0$ and coercive in $\|f'\|$, that is,
\begin{align*}
	\lim_{\substack{\|f'\|\to\infty \\ \|f\|^2 = \lambda}} H(f) = +\infty.
\end{align*}
Moreover, $-\infty < E_{\lambda}^{d_{\mathrm{av}}} \leq 0$.
\end{proposition}

\begin{proof}
	For $2<\gamma_1\leq \gamma_2 \leq 6$ we can, as in Proposition \ref{prop:wellposedness-zero}, estimate the nonlinearity by
	\begin{align*}
		N(f) \lesssim \|f\|_2^{\gamma_1} + \|f\|_2^{\gamma_2}
	\end{align*}
	In case $\gamma_j > 6$ for some $j=1,2$, we can extract the excess part in the $L^{\infty}$ norm, estimating
	\begin{align*}
		\int_{\R} \|T_r f\|_{\gamma}^{\gamma}\,\psi\mathrm{d}r \leq \sup_{r\in\R} \|T_r f\|_{\infty}^{\kappa} \int_{\R} \|T_r f\|_{\gamma-\kappa}^{\gamma-\kappa} \,\psi\mathrm{d}r
	\end{align*}
	for some $2\leq \gamma-\kappa \leq 6$. Using \eqref{eq:1dsobolev},
	\begin{align*}
		\sup_{r\in\R} \|T_r f\|_{\infty} \leq \left(\|f\|_2 \|f'\|_2 \right)^{1/2},
	\end{align*}
	where we used the unitarity of $T_r$ on $L^2$ and the fact that $T_r$ commutes with $\partial_x$, this yields, together with Lemma \ref{lem:st-boundedness},
	\begin{align*}
		N(f) \lesssim \|f'\|_2^{\frac{\kappa_1}{2}} \|f\|_2^{\gamma_1 - \frac{\kappa_1}{2}} + \|f'\|_2^{\frac{\kappa_2}{2}} \|f\|_2^{\gamma_2 - \frac{\kappa_2}{2}},
	\end{align*}
	for suitable $(\gamma_j - 6)_+ \leq \kappa_j \leq \gamma_j - 2$, $j=1,2$, and an implicit constant that can be chosen in such a way that it only depends on the $L^{a_{\delta}}$ norm of $\psi$. It is easy to see that for given $a_{\delta}\geq 1$, one can always choose $\kappa_j <4$.
	Therefore,
	\begin{align}\label{eq:coercivity}
		H(f) \geq \frac{d_{\mathrm{av}}}{2} \|f'\|_2^2 - C \left( \|f'\|_2^{\frac{\kappa_1}{2}} \|f\|_2^{\gamma_1 - \frac{\kappa_1}{2}} + \|f'\|_2^{\frac{\kappa_2}{2}} \|f\|_2^{\gamma_2 - \frac{\kappa_2}{2}} \right)
	\end{align}
	for some constant $C=C(\|\psi\|_{a_{\delta}})$. In particular, if $\|f\|_2^2 = \lambda$, then $H(f) \to \infty$ as $\|f'\|_2 \to \infty$. Moreover,
	\begin{align*}
		E_{\lambda}^{d_{\mathrm{av}}} \geq \inf_{t>0} \left( \frac{d_{\mathrm{av}}}{2} t^2 - C \left( t^{\frac{\kappa_1}{2}} \lambda^{\frac{1}{2}(\gamma_1 - \frac{\kappa_1}{2})} + t^{\frac{\kappa_2}{2}} \lambda^{\frac{1}{2}(\gamma_2 - \frac{\kappa_2}{2})} \right) \right) > -\infty.
	\end{align*}
	
To prove that $E_{\lambda}^{d_{\mathrm{av}}} \leq 0$ we again calculate the energy of suitable centered Gaussians \eqref{eq:gaussian}. Since by \eqref{eq:Vbound},
\begin{align*}
	N(g_{\sigma_0}) \lesssim \|\psi\|_{1} \sup_{r\in\mathrm{supp}\,\psi} \left( \|T_rg_{\sigma_0}\|_{\gamma_1}^{\gamma_1} + \|T_rg_{\sigma_0}\|_{\gamma_2}^{\gamma_2} \right),
\end{align*}
where $2<\gamma_1\leq \gamma_2$, it is not hard to see that
\begin{align*}
	\lim_{\sigma_0\to\infty} H(g_{\sigma_0}) = 0,
\end{align*}
which implies $E_{\lambda}^0 \leq 0$.
\end{proof}

\begin{proof}[Proof of Theorem \ref{thm:positive}]
Fix $0< \lambda <\lambda_{\mathrm{cr}}^{d_{\mathrm{av}}}$. By definition of the threshold and $E_{\lambda}^{d_{\mathrm{av}}}\leq 0$, we must then have $E_{\lambda}^{d_{\mathrm{av}}} = 0$. Assume now that there exists a minimizer $f \in \mathcal{S}_{\lambda}^{d_{\mathrm{av}}}$ with $H(f) = E_{\lambda}^{d_{\mathrm{av}}} = 0$, then
\begin{align}\label{eq:nonexistence}
\begin{split}
	0 &= H(f)
	= \frac{d_{\mathrm{av}}}{2} \|f'\|_2^2 - N(f) = \|f'\|_2^2 \left( \frac{d_{\mathrm{av}}}{2} - \frac{N(f)}{\|f'\|_2^2} \right) \\
	&\geq \|f'\|_2^2 \left( \frac{d_{\mathrm{av}}}{2} - \sup_{\substack{\|g\|_2=1 \\ \|g'\|_2\leq \lambda^{-1/2}\|f'\|_2}} \frac{N(\sqrt{\lambda} g)}{\lambda \|g'\|_2^2} \right) \\
	&= \|f'\|_2^2 \left( \frac{d_{\mathrm{av}}}{2} - R_{\lambda^{-1/2}\|f'\|_2}(\lambda) \right).
\end{split}
\end{align}
	Since $\lambda<\lambda_{\mathrm{cr}}^{d_{\mathrm{av}}}$, Corollary \ref{cor:thresholds} implies that $R_C(\lambda) < \frac{d_{\mathrm{av}}}{2}$ for any $C>0$, in particular for $C=\lambda^{-1/2} \|f'\|_2$, so
	\begin{align*}
		\frac{d_{\mathrm{av}}}{2} - R_{\lambda^{-1/2}\|f'\|_2}(\lambda) > 0,
	\end{align*}
	which by \eqref{eq:nonexistence} implies that $\|f'\|_2 = 0$. But as the kernel of $\partial_x$ is trivial on $H^1(\R)$, we must have $f \equiv 0$, in contradiction to $\|f\|_2^2 = \lambda$, which shows that there cannot exist a minimizer if we are below the threshold $\lambda_{\mathrm{cr}}^{d_{\mathrm{av}}}$.
	
	Assume now $\lambda>\lambda_{\mathrm{cr}}^{d_{\mathrm{av}}}$ and let $(v_n)_{n\in\N} \subset \mathcal{S}_{\lambda}^{d_{\mathrm{av}}}$ be a minimizing sequence for $E_{\lambda}^{d_{\mathrm{av}}}$. Since $H$ is coercive on $\mathcal{S}_{\lambda}^{d_{\mathrm{av}}}$, the sequence $(v_n)$ is bounded. Indeed for $\|v_n\|_2^2 = \lambda$, $H(v_n) \to E_{\lambda}^{d_{\mathrm{av}}}>-\infty$, the bound \eqref{eq:coercivity} implies that $\|v_n'\|_2$ stays bounded, thus also $\|v_n\|_{H^1}$ is bounded \emph{uniformly} in $n\in\N$.
	
	Together with \eqref{eq:1dsobolev} and the unitarity of $T_r$ on $H^1$, we have
	\begin{align*}
		\|T_r v_n\|_{\infty} \leq \|T_r v_n\|_{H^1} = \|v_n\|_{H^1} \leq C_{\lambda}
	\end{align*}
	for any $r\in\mathrm{supp}\,\psi$, and some constant $C_{\lambda}>0$, and Proposition \ref{prop:strictsubadd} implies that the ground state energy $E_{\lambda}^{d_{\mathrm{av}}}$ is strictly sub-additive.
	Hence arguing as in the proofs of \cite[Propositions 4.3 and 4.5]{CHL15} the minimizing sequence is tight modulo shifts and tight in Fourier space, that is there exist shifts $(y_n)_{n\in\N}$ such that for the sequence $w_n:= v_n(\cdot - y_n)$, $n\in\N$, we have
	\begin{align*}
		\lim_{R\to\infty} \sup_{n\in\N} \int_{|x|>R} |w_n(x)|^2\,\mathrm{d}x = 0,
	\end{align*}
	and there exists a constant $K<\infty$ such that for any $L>0$
	\begin{align*}
		\sup_{n\in\N} \int_{|\eta|>L} |\widehat{w_n}(\eta)|^2\,\mathrm{d}\eta = \sup_{n\in\N} \int_{|\eta|>L} |\widehat{v_n}(\eta)|^2\,\mathrm{d}\eta \leq \frac{K}{L^2}.
	\end{align*}
	
	Since $H(w_n) = H(v_n)$ for all $n\in\N$ by translation invariance, $(w_n)_{n\in\N}$ is also a minimizing sequence with $\|w_n\|_2^2 = \|v_n\|_2^2 = \lambda$, which is bounded in $H^1$, $\|w_n\|_{H^1} = \|v_n\|_{H^1} \leq C_{\lambda}$. So the weak compactness of the unit ball implies that there exists a subsequence $w_{n_k}\rightharpoonup v\in H^1$ weakly in $H^1$ and in $L^2$. By tightness, we even have strong convergence in $L^2$. It follows that
	\begin{align*}
		\|v\|_2^2 = \lim_{k\to\infty} \|w_{n_k}\|_2^2 = \lambda >0
	\end{align*}
	and together with the weak sequential lower semi-continuity of the $H^1$ norm this also implies
	\begin{align*}
		\|v'\|_2^2 \leq \liminf_{k\to\infty} \|w_{n_k}'\|_2^2.
	\end{align*}
	Finally, since $\{w_{n_k}\}_{k\in\N}$ is bounded in $H^1$ and converges in $L^2$, the continuity of the nonlinearity $N$ with respect to strong $L^2$-convergence (Lemma \ref{lem:continuity}) yields
	\begin{align*}
		\lim_{k\to\infty} N(w_{n_k}) = N(v).
	\end{align*}
	Altogether, we thus have shown that $H$ is weakly lower semi-continuous along $\{w_{n_k}\}$, in particular
	\begin{align*}
		E_{\lambda}^{d_{\mathrm{av}}} \leq H(v) \leq \liminf_{k\to\infty} H(w_{n_k}) = E_{\lambda}^{d_{\mathrm{av}}},
	\end{align*}
	since $\{w_{n_k}\}$ is minimizing. It follows that $f$ is a minimizer of the variational problem \eqref{eq:varprob}.
The rest of the proof is analogous to the zero average dispersion case $d_{\mathrm{av}} = 0$.
\end{proof}

\appendix

\section{Ekeland's variational principle}\label{app:ekeland}

In this section we briefly derive the following corollary of Ekeland's variational principle \cite[see also the Appendix in \cite{Cos07}]{Eke74} needed in the construction of our modified minimizing sequence. Note that we do not require the functional to be $\mathcal{C}^1$, but only that all its directional derivatives exists and depend linearly and continuously on the direction.

\begin{proposition}\label{prop:ekeland}
	Let $\mathcal{H}$ be a real Hilbert space and $\varphi: \mathcal{H} \to \R$ a continuous functional with the property that all directional derivatives exist and the functional $h \mapsto D_h\varphi(f)$ is linear and continuous for all $f\in \mathcal{H}$.
		
	Assume that $\varphi$ is bounded from below on $\mathcal{S}_{\lambda} = \{ u\in\mathcal{H}: \|u\|^2 = \lambda\}$, and let $(f_n)_{n\in\N} \subset \mathcal{S}_{\lambda}$ be a minimizing sequence for $\varphi|_{\mathcal{S}_{\lambda}}$. Then there exists another minimizing sequence $(g_n)_{n\in\N} \subset \mathcal{S}_{\lambda}$ such that
	\begin{align*}
		\varphi(g_n) \leq \varphi(f_n), \quad \|g_n - f_n\| \to 0
	\end{align*}
	and
	\begin{align*}
		|(D_{h_n} \varphi|_{\mathcal{S}_{\lambda}})(g_n)| \to 0 \quad \text{as} \quad n\to\infty
	\end{align*}
	for any $h_n\in T_{g_n}\mathcal{S}_{\lambda}$ with $\sup_n\|h_n\|<\infty$.
\end{proposition}

\begin{remark}
\begin{enumerate}[label=(\roman*)]
	\item As will be clear from the proof, linearity of the map $h\mapsto D_h\varphi(f)$ is not needed, the only important property is that the one-sided derivatives from left and right coincide, respectively, that $D_{-h}\varphi(f) = -D_h \varphi(f)$ for all $f\in\mathcal{H}$. Linearity allows us to represent, by reflexivity, the directional derivative at a given point $f$ in $\mathcal{S}$ by a vector $\nabla\varphi(f) \in \mathcal{H}$.
	\item Let $u\in\mathcal{S}_{\lambda}$. Since by assumption, the map
		\begin{align*}
			h\mapsto D_h \varphi(u)
		\end{align*}
		is linear and continuous, by the Riesz representation theorem there exists a uniquely determined vector $\nabla\varphi(u)$ such that
		\begin{align*}
			\langle \nabla \varphi(u), h \rangle = D_h \varphi(u).
		\end{align*}
		
		Since $\mathcal{S}_{\lambda}$ is a sphere in  $\mathcal{H}$, we have $\mathcal{H} = T_{u}\mathcal{S}_{\lambda} \oplus \R u$ for all $u\in\mathcal{S}_{\lambda}$. Therefore, the projection of $\nabla\varphi(u)$ onto $T_{u}\mathcal{S}_{\lambda}$ is given by
		\begin{align*}
			\nabla\varphi(u) - \left\langle \nabla\varphi(u), \frac{u}{\|u\|} \right\rangle \frac{u}{\|u\|}.
		\end{align*}
		By Proposition \ref{prop:ekeland}, we thus have
		\begin{align*}
			\left|\left\langle \nabla\varphi(g_n) - \langle \nabla\varphi(g_n), \tfrac{g_n}{\|g_n\|} \rangle \tfrac{g_n}{\|g_n\|}, h_n\right\rangle \right| = \left| (D_{h_n}\varphi|_{\mathcal{S}_{\lambda}})(g_n) \right|\to 0
		\end{align*}
		as $n\to\infty$ for all $h_n\in T_{g_n} \mathcal{S}_{\lambda}$ with $\|h_n\|\le 1$ (and therefore also for all $\widetilde{h}_n \in T_{g_n}\mathcal{S}_{\lambda} \oplus \R g_n =  \mathcal{H}$ with $\|\widetilde{h}_n\|\le 1$), so
		\begin{align*}
		\nabla\varphi(g_n) - \langle \nabla\varphi(g_n), \tfrac{g_n}{\|g_n\|} \rangle \tfrac{g_n}{\|g_n\|} \to 0, \quad n\to\infty
		\end{align*}
		strongly in $\mathcal{H}$.	
\end{enumerate}
\end{remark}

\begin{proof}
	Let $c = \inf_{\mathcal{S}_{\lambda}} \varphi$ and set $\epsilon_n = \max\left\{ \frac{1}{n}, \varphi(f_n) - c\right\}$. By Ekeland's variational principle there exists a sequence $(g_n)_{n\in\N}\subset \mathcal{S}_{\lambda}$ such that $\varphi(g_n) \leq \varphi(f_n)$ for all $n\in \N$, $\|g_n - f_n\| \to 0$ as $n\to\infty$, and
	\begin{align}\label{eq:ekeland}
		\varphi(g_n) < \varphi(u) + \sqrt{\epsilon_n} \, \|g_n - u\| \quad \text{for all} \quad u \neq g_n.
	\end{align}
	Now let $\gamma:(-1,1) \to \mathcal{S}_{\lambda}$ be a $\mathcal{C}^1$ curve with $\gamma(0) = g_n$ and $\gamma'(0) = h_n$, for some arbitrary $h_n \in T_{g_n}\mathcal{S}_{\lambda}$.
	Then, by means of the continuity of $h\mapsto D_h \varphi(f)$ for all $f\in\mathcal{H}$, we have
	\begin{align*}
		\lim_{t\to 0} \frac{\varphi(\gamma(t)) - \varphi(\gamma(0))}{t} &= \lim_{t\to 0} \frac{\varphi(\gamma(0) + t \gamma'(0) + {o}(t)) - \varphi(\gamma(0))}{t} \\
		&= \lim_{t\to 0} \frac{\varphi(\gamma(0)) + t D_{\gamma'(0) + t^{-1}{o}(t)} \varphi(\gamma(0)) + o(t) - \varphi(\gamma(0))}{t} \\
		&= \lim_{t\to 0} D_{\gamma'(0)+ t^{-1}o(t)}\varphi(\gamma(0)) = D_{\gamma'(0)} \varphi(\gamma(0)) = D_{h_n}\varphi(g_n),
	\end{align*}
	As the curve $\gamma$ was arbitrary, this implies
	\begin{align*}
		(D_{h_n}  \varphi|_{S_{\lambda}})(g_n) = \lim_{t\to 0} \frac{\varphi(\gamma(t)) - \varphi(\gamma(0))}{t}.
	\end{align*}
	By \eqref{eq:ekeland}, for all $t>0$ we have
	\begin{align*}
		\varphi(\gamma(t))-\varphi(\gamma(0))> -\sqrt{\epsilon_n} \|\gamma(0) - \gamma(t)\|,
	\end{align*}
	and dividing by $t>0$ and letting $t\to 0$ yields
	\begin{align*}
		(D_{h_n}  \varphi|_{S_{\lambda}})(g_n) = \lim_{t\downarrow 0} \frac{\varphi(\gamma(t)) - \varphi(\gamma(0))}{t} \geq -\sqrt{\epsilon_n}   \|\gamma'(0)\| = - \sqrt{\epsilon_n}\|h_n\|
	\end{align*}
	Similarly, exchanging $t$ by $-t$, one obtains
	\begin{align*}
		(D_{h_n}  \varphi|_{S_{\lambda}})(g_n) = \lim_{t\downarrow 0} \frac{\varphi(\gamma(-t)) - \varphi(\gamma(0))}{-t} \leq \sqrt{\epsilon_n}  \|\gamma'(0)\| = \sqrt{\epsilon_n}\|h_n\|,
	\end{align*}
	and therefore
	\begin{align*}
		\left|(D_{h_n}  \varphi|_{S_{\lambda}})(g_n)\right| \leq \sqrt{\epsilon_n}\|h_n\| \to 0 \quad \text{as} \quad n \to \infty.
	\end{align*}
\end{proof}
\section*{Acknowledgements}
\noindent
Y.-R.L.\ and V.Z.\ thank the Department of Mathematics at KIT, D.H.\ thanks the Department of Mathematics at Sogang University, and T.R.\ thanks the Department of Mathematics at the University of Illinois at Urbana-Champaign and the School of Mathematics at Georgia Institute of Technology for their warm hospitality.

D.H.\ and T.R.\ gratefully acknowledge financial support by the Deutsche Forschungsgemeinschaft (DFG) through CRC 1173 `Wave Phenomena'.
D.H.\ also thanks the Alfried Krupp von Bohlen und Halbach Foundation for financial support.
Y.-R.L.\  thanks the National Research Foundation of Korea (NRF) for financial support funded by the Ministry of Education (No. 2014R1A1A2058848).
V.Z.\  thanks Simons foundation for partial support (\#278840 to Vadim Zharnitsky). \\

\bigskip
\end{document}